\documentclass[a4paper,11pt]{amsart}

\usepackage{amssymb,amsmath, pinlabel, mathabx, amscd, tikz, xcolor}

  \usepackage{enumitem}
 \usepackage{lineno} 
  \usepackage[normalem]{ulem} 
 
 \usepackage[pagebackref]{hyperref}
  \hypersetup{
colorlinks=true,
linkcolor=cyan,
citecolor=cyan
}

\usepackage{mathtools}
\usepackage[pagebackref]{hyperref}

\input xy
\xyoption{all}

\usepackage{color}

\newtheorem{theorem}{Theorem}[section]
\newtheorem{lemma}[theorem]{Lemma}
\newtheorem{proposition}[theorem]{Proposition}

\theoremstyle{definition}
\newtheorem{definition}[theorem]{Definition}

\newtheorem{example}[theorem]{Example}

\newcommand{\gras}[1]{{\mathbb #1}} 

\newcommand{\N}{\gras{N}}
 
\newcommand{\Q}{\gras{Q}} 
\newcommand{\R}{\gras{R}} 
\newcommand{\C}{\gras{C}}

\newcommand{\cB}{\mathcal{B}}

\newcommand{\cE}{\mathcal{E}}
\newcommand{\cF}{\mathcal{F}}
\newcommand{\cG}{\mathcal{G}}
\newcommand{\cH}{\mathcal{H}}

\newcommand{\cO}{\mathcal{O}}
\newcommand{\cP}{\mathcal{P}}

\newcommand{\cV}{\mathcal{V}}

\newcommand{\m}{\frak{m}}

\def\elem(#1,#2){  \{ \frac{#1}  {\overline {\ #2\ }} \} }

\title[Ultrametrics and surface singularities]{Ultrametrics and surface singularities}
\author{Patrick Popescu-Pampu}
   \address{Univ. Lille, CNRS, UMR 8524 - Laboratoire Paul Painlev\'e, F-59000 Lille, 
         France.}
   \email{patrick.popescu-pampu@univ-lille.fr}
\date{13 July 2020}

\subjclass[2010]{14B05 (primary), 14J17, 32S25}
\keywords{Arborescent singularity,  Birational geometry, Block, 
  Cut-vertex, Eggers-Wall tree, Intersection number, Newton-Puiseux series, Normal surface singularity, 
  Resolution of singularities, Tree, Ultrametric, Valuation}

    \makeindex

\begin{document}

{\bf To appear in the volume \emph{Introduction to Lipschitz geometry of singularities}, \\
   edited by W. Neumann and A. Pichon, Springer, 2020 or 2021.}
   \bigskip

\begin{abstract}
     The present lecture notes give an introduction to works of Garc\'{\i}a Barroso, 
     Gonz\'alez P\'erez, Ruggiero and the author. The starting point 
     of those works is a theorem of P\l oski, stating that one defines an ultrametric 
     on the set of branches drawn on a smooth surface singularity by associating to any pair 
     of distinct branches the quotient of the product of their multiplicities by their 
     intersection number. We show how to construct ultrametrics 
     on certain sets of branches drawn on any normal surface singularity from their 
     mutual intersection numbers and how to interpret the associated rooted trees  
     in terms of the dual graphs of adapted embedded resolutions.
     The text begins by recalling basic properties 
     of intersection numbers and multiplicities on smooth surface singularities and 
     the relation between ultrametrics on finite sets and rooted trees. 
     On arbitrary normal surface singularities one has to use 
     Mumford's definition of intersection numbers of curve singularities drawn on them, 
     which is also recalled. 
\end{abstract}

\maketitle

\tableofcontents

\section*{Introduction}
\label{sec:intro}

This paper is an expansion of my notes 
prepared for the course with the same title given  
at the \emph{International school on singularities and Lipschitz geometry},  
which took place in Cuernavaca (Mexico) from June 11th to 22nd 2018. 
\medskip

If $S$ denotes a \emph{normal surface singularity}, that is, a germ of normal complex 
analytic surface, a \emph{branch} on it is an irreducible germ of analytic curve contained in $S$. 
In his 1985 paper \cite{P 85}, Arkadiusz P\l oski proved that if one associates 
to every pair of distinct branches on 
the singularity $S=(\C^2, 0)$ the quotient $\dfrac{A \cdot B}{m(A) \cdot m(B)}$ 
of their intersection number by 
the product of their multiplicities, then for every triple of pairwise distinct branches, two of 
those quotients are equal and the third one is not smaller than them. An equivalent formulation is
that the inverses $\dfrac{m(A) \cdot m(B)}{A \cdot B}$ of the previous quotients 
define an \emph{ultrametric} on the set of branches on $(\C^2, 0)$. 

Using the facts  
that the multiplicity of a branch is equal to its intersection number with a smooth branch $L$ 
transversal to it, and that a given function is an ultrametric on a set if and only if 
it is so in restriction to all its finite subsets, one deduces that P\l oski's theorem is 
a consequence of:

\medskip
\begin{quote}
      {\bf Theorem A.} {\em Let $L$ be a smooth branch on the smooth surface singularity 
      $S$ and let $\cF$ 
      be a finite set of branches on $S$, transversal to $L$. Then the function 
      $u_L : \cF \times \cF \to [0, \infty)$ defined by 
      $u_L(A, B) := \dfrac{(L \cdot A) \cdot (L \cdot B)}{A \cdot B}$ if $A \neq B$ 
      and $u_L(A, A) := 0$ is an ultrametric on $\cF$. }
\end{quote}
\medskip

 This may be seen as a property of the pair $(S, L)$ and one may ask whether it 
 extends to other pairs consisting of a normal surface singularity and a branch on it. It turns out 
 that this property characterizes the so-called \emph{arborescent singularities}, that is, 
 the normal surface singularities such that the dual graph of every good resolution is a tree. 
 Namely, one has the following 
 theorem, which combines \cite[Thm. 85]{GBGPPP 18} and \cite[Thm. 1.46]{GBGPPPR 19}:
 
 \medskip
 \begin{quote}
      {\bf Theorem B.} {\em Let $L$ be a branch on the normal surface singularity $S$. Then the function 
      $u_L$ defined as before 
      is an ultrametric on any finite set $\cF$ of branches on $S$ distinct from $L$ 
      if and only if $S$ is an arborescent singularity. }
\end{quote}
\medskip

It is possible to think topologically about ultrametrics on finite sets in terms of certain 
types of decorated rooted trees.  In particular, any such ultrametric determines 
a rooted tree. One may try to describe this tree directly from the pair $(S,\cF \cup \{L\})$,  
when $S$ is arborescent and the ultrametric 
is the function $u_L$ associated to a branch $L$ on it. 
In order to formulate such a description, we need the notion of \emph{convex hull} of 
a finite set of vertices of a tree: it is the union of the paths joining those vertices pairwise. 

The following result was obtained in \cite[Thm. 87]{GBGPPP 18}:

\medskip
 \begin{quote}
      {\bf Theorem C.} {\em Let $L$ be a branch on the arborescent singularity $S$ and let 
      $\cF$ be a finite set of branches on $S$ distinct from $L$. Then the 
      rooted tree determined by the ultrametric $u_L$ on $\cF$ is isomorphic 
      to the convex hull of the strict transform of $\cF \cup \{L\}$ in the dual 
      graph of its preimage by an embedded resolution of it, rooted at the vertex 
      representing the strict transform of $L$. }
\end{quote}
\medskip

Even when the singularity $S$ is not arborescent, the function $u_L$ becomes an ultrametric 
in restriction to suitable sets $\cF$ 
of branches on $S$. Those sets are defined only in terms of convex hulls 
taken in the so-called \emph{brick-vertex tree} of the dual graph of an embedded resolution 
of $\cF \cup \{L\}$, and do not depend 
on any numerical parameter of the exceptional divisor of the resolution, 
be it a genus or a self-intersection number. The brick-vertex tree 
of a connected graph is obtained canonically by replacing each \emph{brick} 
-- a maximal inseparable subgraph which is not an edge -- by a star, whose central vertex 
is called a \emph{brick-vertex}. One has the following generalization of Theorem C 
(see \cite[Thm. 1.42]{GBGPPPR 19}):

\medskip
 \begin{quote}
      {\bf Theorem D.} {\em Let $L$ be a branch on the normal surface singularity $S$ and let 
      $\cF$ be a finite set of branches on $S$ distinct from $L$.  Consider an  
      embedded resolution of $\cF \cup \{L\}$. Assume that the convex hull of its 
      strict transform in the brick-vertex tree of the dual graph of its preimage does not contain 
      brick-vertices of valency at least $4$ in the convex hull. 
      Then the function $u_L$ is an ultrametric in restriction to $\cF$ 
      and the associated rooted tree is isomorphic 
      to the previous convex hull, rooted at the vertex 
      representing the strict transform of $L$. }
\end{quote}
\medskip

If $S$ is not arborescent, 
there may exist other sets of branches on which $u_L$ restricts to an ultrametric. 
Unlike the sets described in the previous theorem, in general they do not depend only on the topology 
of the dual graph of their preimage on some embedded resolution, but also  
on the self-intersection numbers of the components of the exceptional divisor 
(see \cite[Ex. 1.44]{GBGPPPR 19}). 

The aim of the present notes is to introduce the reader to the previous results. 
Note that in the article \cite[Part 2]{GBGPPPR 19}, these results were extended to the space 
of real-valued semivaluations of the local ring of $S$. 

Let us describe briefly the structure of the paper. 
In Section \ref{sec:intmult} are recalled basic facts about \emph{multiplicities} and 
\emph{intersection numbers} of plane curve singularities. In Section \ref{sec:statmt} 
are stated two equivalent formulations of P\l oski's theorem. In Section 
\ref{sec:ultramtrees} is explained the relation between ultrametrics and rooted trees 
mentioned above, an intermediate concept being that of \emph{hierarchy} on a finite set. 
Using this relation,  Section \ref{sec:proofthm} presents a proof of Theorem A. 
This proof uses the so-called \emph{Eggers-Wall tree} of a plane curve singularity 
relative to a smooth reference branch $L$, 
constructed using associated Newton-Puiseux series. 
Section \ref{sec:ultramchar} explains the notions used in the formulation of 
Theorem B, that is, those of \emph{good resolution}, \emph{embedded resolution}, 
associated \emph{dual graph} and 
\emph{arborescent singularity}. In Section \ref{sec:brickvertex} are described the 
related notions of \emph{cut-vertex} and 
\emph{brick-vertex tree} of a finite connected graph. Section \ref{sec:ultrambvtree} 
explains and illustrates the statement of Theorem D. 
In Section \ref{sec:mumfint} is explained Mumford's intersection theory of 
divisors on normal surface singularities, after a proof of a fundamental property of such 
singularities, stating that the intersection form of any of their resolutions is negative 
definite. In Section \ref{sec:refineq} the ultrametric inequality concerning the restriction 
of $u_L$ to a triple of branches is reexpressed in terms of the notion of 
\emph{angular distance} on the dual graph of an adapted resolution. 
A crucial property of this distance is stated, which relates it to the cut-vertices of the 
dual graph. In Section \ref{sec:thmgraph} is sketched the proof of a theorem of 
pure graph theory, relating distances satisfying the previous crucial property and 
the brick-vertex tree of the graph. This theorem implies Theorem D.

\medskip
{\bf Acknowledgements.} 
This work was partially supported by the French grants ANR-17-CE40-0023-02 LISA 
and Labex CEMPI (ANR-11-LABX-0007-01). 
I am grateful to Alexandre Fernandes, 
   Adam Parusinski, Anne Pichon, Maria Ruas, Jos\'e Seade and Bernard Teissier, 
   who formed the scientific committee of the 
   \emph{International school on singularities and Lipschitz geometry}, for having invited me to give a 
   course on the relations between ultrametrics and surface singularities. 
   I am very grateful to my co-authors Evelia Garc\'{\i}a Barroso, Pedro Gonz\'alez P\'erez 
   and Matteo Ruggiero for the collaboration leading to our works \cite{GBGPPP 18}, 
   \cite{GBGPPPR 19} presented in this paper and for their remarks on a previous version 
   of it.

\section{Multiplicity and intersection numbers for plane curve singularities}
\label{sec:intmult}

In this section we recall the notions of \emph{multiplicity} of a \emph{plane curve singularity} 
and \emph{intersection number} of two such singularities. One may find more details in   
\cite[Sect. 5.1]{JP 00} or \cite[Chap. 8]{F 01}. 
\medskip

Let $(S,s)$ be a {\bf smooth surface singularity}, \index{singularity!smooth surface} that is, 
a germ of \emph{smooth} complex analytic surface. Denote by  $\boxed{\cO_{S,s}}$  
its local $\C$-algebra and by $\boxed{\m_{S,s}}$ its maximal ideal, containing 
the germs at $s$ of holomorphic functions vanishing at $s$. 

A {\bf local coordinate system} \index{coordinate system!local} 
on $S$ at $s$ is a pair $(x,y) \in \m_{S,s} \times  \m_{S,s}$ 
establishing an isomorphism between a neighborhood of $s$ in $S$ and a 
neighborhood of the origin in $\C^2$. Algebraically speaking, this is equivalent to 
the fact that the pair $(x,y)$ generates the maximal ideal $\m_{S,s}$, or that it 
realizes an isomorphism $\cO_{S,s} \simeq \C \{ x, y \}$. This isomorphism allows 
to see each germ $f \in \cO_{S,s}$ as a convergent power series in the variables 
$x$ and $y$. 

A {\bf curve singularity on $(S,s)$} \index{singularity!curve} 
is a germ $(C,s) \hookrightarrow (S,s)$ of not necessarily reduced 
curve on $S$, passing through $s$. As the germ $(S,s)$ is isomorphic to the germ of the affine 
plane $\C^2$ at any of its points, one says also that $(C,s)$ is a {\bf plane curve singularity}. 
A {\bf defining function} of $(C,s)$ is a function 
$f \in \m_{S,s}$ such that $\cO_{C,s} = \cO_{S,s}/(f)$, where $(f)$ denotes the principal 
ideal of $\cO_{S,s}$ generated by $f$. We write then $C = \boxed{Z(f)}$. 

The curve singularity $(C,s)$ may also be seen as an effective 
principal divisor on $(S,s)$. This allows to write $C = \sum_{i \in I} p_i C_i$, where  
$p_i \in \N^*$ for all $i \in I$ and the curve singularities $C_i$ are pairwise distinct 
and irreducible. We say in this case that the $C_i$'s are the {\bf branches of $C$}. A {\bf branch 
\index{branch} on $(S,s)$} is an irreducible curve singularity on $(S,s)$. 

Next definition introduces the simplest invariant of a plane curve singularity: 

\begin{definition}  \label{def:mult}
     Assume that $f \in \cO_{S,s}$. Its {\bf multiplicity} \index{multiplicity!of a curve singularity} 
     is the vanishing order 
     of $f$ at $s$:
        \[  \boxed{m_s(f)} := \sup \{ n \in \N, \: f \in \m_{S,s}^n \}   \in \N \cup \{\infty\}. \]
      If $(C,s)$ is the curve singularity defined by $f$, we say also that $\boxed{m_s(C)} := m_s(f)$ is 
      its {\bf multiplicity} at $s$. 
\end{definition}

It is a simple exercise to check that the multiplicity of a curve singularity is  
independent of the function defining it. If one chooses local coordinates 
$(x,y)$ on $(S,s)$, then $m_s(f)$ is the smallest degree of the monomials appearing in 
the expression of $f$ as a convergent power series in the variables $x$ and $y$. 
One has $m_s(f) = \infty$ if and only if $f = 0$ and  $m_s(f)=1$ if and 
only if $f$ defines a \emph{smooth} branch on $(S,s)$.  

The following definition describes a measure of the way in which two 
curve singularities intersect:

\begin{definition}  \label{def:intnumb}
    Let $C, D \hookrightarrow (S,s)$ be two plane curve singularities defined by 
    $f,g \in \m_{S,s}$. Then their {\bf intersection number} \index{intersection number} is defined by:
         \[  \boxed{C \cdot D} := \dim_{\C} \frac{\cO_{S,s}}{(f,g)}    \in \N \cup \{ \infty \}, \]
    where $(f,g)$ denotes the ideal of $\cO_{S,s}$ generated by $f$ and $g$. 
\end{definition}

Note that $C \cdot D < + \infty$ if and only if $C$ and $D$ do not share common branches, 
which is also equivalent to the existence of $n \in \N^*$ such that one has the following 
inclusion of ideals: $(f,g) \supseteq \m_{S,s}^n$. Nevertheless, unlike the multiplicity, 
the intersection number $C \cdot D$ is not always equal to the smallest exponent $n$ having this 
property. For instance,  if one takes $f := x^3$ and $g := y^2$, then 
$C \cdot D = 6$ but $(f,g) \supseteq (x,y)^5$. We leave 
the verification of the previous facts as an exercise.

The following proposition, which may be proved using Proposition 
\ref{prop:intpar} below, relates multiplicities and intersection numbers: 

\begin{proposition}   \label{prop:ineqint}
     If $(C,s) \hookrightarrow (S,s)$ is a plane curve singularity, then 
     $C \cdot L \geq m_s(C)$ for any smooth branch $L$ through $s$, with equality 
     if and only if $L$ is transversal to $C$. More generally, if $D$ is a second 
     curve singularity on $(S,s)$, then $C \cdot D \geq m_s(C) \cdot m_s(D)$, 
     with equality if and only if $C$ and $D$ are transversal. 
\end{proposition}

Let us explain the notion of {\em transversality} used in the previous proposition, as it is 
more general than the standard notion of transversality, which applies only to smooth   
submanifolds of a given manifold. If $C$ is a branch on $(S,s)$ and one chooses a
local coordinate system $(x,y)$ on $(S,s)$, as well as a defining function $f$ of $C$, 
it may be shown that the lowest degree part of $f$ is a power of a complex linear form 
in $x$ and $y$. This linear form defines a line in the tangent plane $T_s S$ of $S$ 
at $s$, which is by definition the {\bf tangent line} \index{tangent line} of $C$ at $s$. One may show that 
it is independent of the choices of local coordinates and defining function of $C$. 
If $C$ is now an arbitrary curve singularity, then its {\bf tangent cone} 
\index{tangent cone!of a plane curve singularity} 
is the union of 
the tangent lines of its branches. Given two plane curve singularities on the same 
smooth surface singularity $S$, one says that 
they are {\bf transversal} \index{transversal!plane curve singularities} 
if each line of the tangent cone of one of them is transversal 
(in the classical sense) to each line of the tangent cone of the other one.

Let us pass now to the question of \emph{computation} of intersection numbers. A basic method 
consists in breaking the symmetry between the two curve singularities, by 
working with a defining function of one of them and by \emph{parametrizing} the other one. 
One has to be cautious and choose a \emph{normal parametrization}, in the 
following sense: 

\begin{definition}  \label{def:normpar}
      A {\bf normal parametrization} \index{normal parametrization} 
      of the branch $(C,s)$ is a germ of holomorphic morphism 
      $\nu : (\C, 0) \to (C, s)$ which is a normalization morphism, that is, it has topological 
      degree one. 
\end{definition}

For instance, if the branch $(C,0)$ on $(\C^2, 0)$ is defined by the function $y^2 - x^3$, then 
$t \to (t^2, t^3)$ is a normal parametrization of $C$, but $u \to (u^4, u^6)$ is not. 
A normal parametrization of a branch $(C,s)$ may be also characterized by asking it 
to establish a homeomorphism between suitable representatives of the germs 
$(\C, 0)$ and $(C, s)$. 

Normalization morphisms may be defined more generally for reduced germs $(X,x)$ 
of arbitrary dimension (see \cite[Sect. 4.4]{JP 00}), 
by considering the multi-germ whose multi-local ring is the integral closure 
of the local ring $\cO_{X,x}$ in its total ring of fractions. Except for curve singularities, 
the source of a normalization morphism is not smooth in general.

The following proposition is classical and states the announced expression of intersection 
numbers in terms of a parametrization of one germ and a defining function of the 
second one (see \cite[Prop. II.9.1]{BHPV 04} or \cite[Lemma 5.1.5]{JP 00}):

\begin{proposition}  \label{prop:intpar}
      Let $C$ be a branch on the smooth surface singularity 
      $(S,s)$ and $D$ be a second curve singularity, not necessarily reduced. 
      Let $\nu : (\C, 0) \to (C, s)$  be a normal parametrization 
      of $C$ and let $g \in \m_{S,s}$ be a defining function of $D$. Then:
          \[  C \cdot D = \mathrm{ord}_t \left( g \circ \nu(t)  \right), \]
       where $\boxed{\mathrm{ord}_t}$ denotes the order of a power series 
       in the variable $t$.
\end{proposition}

\begin{proof} 
   This proof is adapted from that of \cite[Lemma 5.1.5]{JP 00}. 
   
   The order of the zero power series is equal to $\infty$ by definition, therefore 
    the statement is true when $C$ is a branch of $D$. 
    
    Let us assume from now on that {\em $C$ is not a branch of $D$}. 
   
   Consider a defining function $f \in  \m_{S,s}$ of $C$. By Definition \ref{def:intnumb}: 
     \begin{equation} \label{eq:firstcomp}
          C \cdot D = \dim_{\C} \frac{\cO_{S,s}}{(f,g)} = \dim_{\C} \frac{\cO_{S,s}/(f)}{(g_C)} = 
          \dim_{\C} \frac{\cO_{C,s}}{(g_C)}, 
      \end{equation}
     where we have denoted by $g_C \in \cO_{C,s}$ the restriction of $g$ to the branch $C$. 
     
     Algebraically, the normal parametrization $\nu : (\C, 0) \to (C, s)$ corresponds to a 
     morphism of local $\C$-algebras $\cO_{C,s} \hookrightarrow \C\{ t\}$, isomorphic to 
     the inclusion morphism of $\cO_{C,s}$ into its integral closure taken inside its quotient field. 
     In order to distinguish them, denote from now on by $g_C \:  \cO_{C,s}$ the principal ideal 
     generated by $g_C$ inside $\cO_{C,s}$ and by $g_C \:  \C\{t\}$ its analog inside $\C\{t\}$.      
     One has the following equality inside the local $\C$-algebra $\C\{t\}$:
        \[  g \circ \nu(t)  = g_C.  \]
      As a consequence: 
         \[  g_C \:   \C\{t\} =  t^{\mathrm{ord}_t \left( g \circ \nu(t)  \right)} \C\{t\}.  \]
     Therefore: 
        \begin{equation} \label{eq:secomp}
           \mathrm{ord}_t \left( g \circ \nu(t)  \right) 
                 = \dim_{\C} \frac{\C\{t\}}{g_C \:  \C\{t\}}.
         \end{equation}
      By comparing equations (\ref{eq:firstcomp}) and (\ref{eq:secomp}), 
      we see that the desired equality is equivalent to:
          \begin{equation} \label{eq:thirdcomp}
           \dim_{\C} \frac{\cO_{C,s}}{g_C \:  \cO_{C,s}}
                 = \dim_{\C} \frac{\C\{t\}}{g_C \:  \C\{t\}}.
         \end{equation}
       The two quotients appearing in (\ref{eq:thirdcomp}) are the cokernels of the two 
       injective multiplication maps $\cO_{C,s} \overset{\cdot g_C}{\longrightarrow} \cO_{C,s}$ 
       and $\C\{t\} \overset{\cdot g_C}{\longrightarrow} \C\{t\}$.  The associated short exact 
       sequences may be completed into a commutative diagram in which the first two vertical 
       maps are the inclusion map $\cO_{C,s} \hookrightarrow \C\{ t\}$:
       \begin{center}
            \begin{tikzpicture}[node distance=2cm, auto]  
                    \node (O1) {$0$};
                    \node   (A) [right of=O1] {$\cO_{C,s}$}; 
                    \draw[->] (O1) to node {\:} (A);
                    \node (B) [right of=A] {$\cO_{C,s}$}; 
                    \draw[->] (A) to node {\:} (B);
                    \node (C) [right of=B] {$ \dfrac{\cO_{C,s}}{g_C \:  \cO_{C,s}}$}; 
                      \draw[->] (B) to node {\:} (C);
                     \node (O2) [right of=C] {$0$};
                     \draw[->] (C) to node {\:} (O2);
                      
                     \node  (O3) [below of=O1] {$0$};
                      \node   (A') [right of=O3] {$\C\{t\}$}; 
                    \draw[->] (O3) to node {\:} (A');
                    \node (B') [right of=A'] {$\C\{t\}$}; 
                    \draw[->] (A') to node {\:} (B');
                    \node (C') [right of=B'] {$ \dfrac{\C\{t\}}{g_C \:  \C\{t\}}$}; 
                      \draw[->] (B') to node {\:} (C');
                     \node (O4) [right of=C'] {$0$};
                     \draw[->] (C') to node {\:} (O4);
                     
                     \draw[->] (A) to node {\:} (A');
                     \draw[->] (B) to node {\:} (B');
                     \draw[->] (C) to node {\:} (C');
              \end{tikzpicture}   
     \end{center} 
   The last vertical map is not necessarily an isomorphism. We want to show that its source 
   and its target have the same dimension. Let us complete it into an exact sequence by 
   considering its kernel $K_1$ and cokernel $K_2$:
       \[ 0 \longrightarrow K_1  \longrightarrow   \dfrac{\cO_{C,s}}{g_C \:  \cO_{C,s}} 
             \longrightarrow   \dfrac{\C\{t\}}{g_C \:  \C\{t\}} \longrightarrow  K_2    \longrightarrow 0.  \]  
      For every finite exact sequence of finite-dimensional vector spaces, the alternating sum 
      of dimensions vanishes. Therefore:
         \[ \dim_{\C} K_1 -   \dim_{\C} \frac{\cO_{C,s}}{g_C \:  \cO_{C,s}}
                 +  \dim_{\C} \frac{\C\{t\}}{g_C \:  \C\{t\}} - \dim_{\C} K_2 =0. \] 
       This shows that the desired equality (\ref{eq:thirdcomp}) would result from 
       the equality $\dim_{\C} K_1 =  \dim_{\C} K_2$. This last equality is a consequence 
       of the so-called ``snake lemma'' (see for instance \cite[Prop. 2.10]{AM 69}), 
       applied to the previous commutative diagram. Indeed, by this lemma, one has an  
       exact sequence:
           \[ 0 \longrightarrow K_1  \longrightarrow   \dfrac{\C\{t\} }{\cO_{C,s}} 
             \longrightarrow  \dfrac{\C\{t\} }{\cO_{C,s}}  \longrightarrow  K_2    \longrightarrow 0.   \]
        Reapplying the previous argument about alternating sums of dimensions, one gets 
        the needed equality $\dim_{\C} K_1 =  \dim_{\C} K_2$.
\end{proof}

Note that the previous proof shows in fact that for any abstract branch $(C,s)$, 
not necessarily planar, one has the equality:
   \begin{equation} \label{eq:genhyp} 
       \dim_{\C} \frac{\cO_{C,s}}{(g)}  = \mathrm{ord}_t \left( g \circ \nu(t)  \right),
   \end{equation}
for any $g \in \cO_{C,s}$ and for any normal parametrization $\nu : (\C, 0) \to (C,s)$ 
of $(C,s)$. If the branch $(C,s)$ is contained in an ambient germ 
$(X,s)$ and $H$ is an effective principal divisor on $(X,s)$ which does not contain 
the branch, then equality (\ref{eq:genhyp}) shows that 
the intersection number of $C$ and $H$ at $s$ may be computed  
as the order of the series obtained by composing a defining function of $(H,s)$ and a 
normal parametrization of $(C,s)$.

\begin{example}  \label{ex:basic}
    Consider the branches:
        $$ \left\{ \begin{array}{l}
                         A := Z(y^2 - x^3), \\
                         B := Z(y^3 - x^5) , \\
                         C := Z(y^6 - x^5)
                       \end{array} \right.$$
     on the smooth surface singularity $(\C^2, 0)$. 
     Denoting by $m_0$ the multiplicity function at the origin of $\C^2$, we have: 
            $$m_0(A) =2, \: m_0(B) = 3, \:  m_0(C)= 5,$$ 
    as results from Definition \ref{def:mult}. 
     Using Proposition \ref{prop:intpar} and the fact that whenever $m$ and $n$ are coprime 
     positive integers, 
     $t \to (t^n , t^m)$ is a normal parametrization of $Z(y^n - x^m)$, one gets the following 
     values for the intersection numbers of the branches $A, B, C$: 
         $$ B \cdot C = 15, \:  C \cdot A = 10, \:  A \cdot B = 9.$$
     Therefore: 
          $$ \left\{ \begin{array}{l}
                         \dfrac{B \cdot C}{m_0(B) \cdot m_0(C)} = 1, \\
                            \\
                        \dfrac{C \cdot A}{m_0(C) \cdot m_0(A)} = 1, \\
                            \\
                         \dfrac{A \cdot B}{m_0(A) \cdot m_0(B)} =  \dfrac{3}{2}.
                       \end{array} \right.$$
     One notices that two of the previous quotients are equal and the third one is 
     greater than them. P\l oski discovered that this is a general phenomenon for 
     plane branches, as explained in the next section. 
\end{example}

\section{The statement of P\l oski's theorem}
\label{sec:statmt}

In this section we state a theorem of P\l oski of 1985 and a reformulation of it in terms 
of the notion of \emph{ultrametric}. 
\medskip

Denote simply by $\boxed{S}$ the germ of \emph{smooth} surface $(S,s)$ and 
by $\boxed{m(A)}$ the multiplicity of a branch $(A,s) \hookrightarrow (S,s)$. 

In his 1985 paper \cite{P 85}, P\l oski \index{P\l oski's theorem} proved the following theorem:

\begin{theorem}  \label{thm:ploski}
      If $A, B, C$ are three pairwise distinct branches on a smooth surface singularity $S$, 
      then one has the following relations, up to a permutation of the three fractions: 
        $$\dfrac{A \cdot B}{m(A) \cdot m(B)}  \geq  \dfrac{B \cdot C}{m(B) \cdot m(C)} = 
                \dfrac{C \cdot A}{m(C) \cdot m(A)}. $$
\end{theorem}

Denote by $\boxed{\mathcal{B}(S)}$ the infinite set of branches on $S$. 
By inverting the fractions appearing in the statement of Theorem \ref{thm:ploski}, 
it may be reformulated in the following equivalent way:

\begin{theorem}  \label{thm:ploskibis}
       Let $S$ be a smooth surface singularity. 
       Then the map $\mathcal{B}(S) \times \mathcal{B}(S) \to [0, \infty)$ defined by 
         $$ (A, B) \to \left\{ \begin{array}{ll}
                                           \dfrac{m(A) \cdot m(B)}{A \cdot B} & \mbox{ if } A \neq B, \\
                                           0 & \mbox{ otherwise}
                                        \end{array} \right. $$
      is an ultrametric. 
\end{theorem}

What does it mean that a function is an ultrametric? We explain this in the next section and 
we show how to think topologically about ultrametrics on finite sets in terms of 
certain kinds of decorated rooted trees. This way of thinking is used then in Section \ref{sec:proofthm} 
in order to prove the reformulation \ref{thm:ploskibis} of P\l oski's theorem.

\section{Ultrametrics and rooted trees}
\label{sec:ultramtrees}

In this section we define the notion of \emph{ultrametric} and we explain how to think 
about an ultrametric on a finite set in topological terms, as a special kind of 
rooted and decorated tree. This passes through understanding that the closed 
balls of an ultrametric form a \emph{hierarchy} and that finite hierarchies 
are equivalent to special types of decorated rooted trees. For more details, one may consult 
\cite[Sect. 3.1]{GBGPPP 18}. 
\medskip

\begin{definition}  \label{def:ultram}  \index{distance!ultrametric}
     Let $(M,d)$ be a metric space. It is called {\bf ultrametric} \index{ultrametric} \index{ultrametric!space} 
     if one has the following strong 
     form of the triangle inequality:
          $$d(A,B) \leq \max \{ d(A,C), d(B, C) \}, \mbox{ for all } A, B, C \in M. $$
      In this case, one says also that $d$ is {\bf an ultrametric} on the set $M$. 
\end{definition}

In any metric space $(M, d)$, a {\bf closed ball} is a subset of $M$ of the form:
    \[  \boxed{\cB(A, r)} := \{P \in M, \:  d(P, A) \leq r\}\]
 where the {\bf center} $A \in M$ and the {\bf radius} $r \in [0, \infty)$ are given. 
 As we will see shortly, given a closed ball, neither its center nor its radius 
 are in general well-defined, contrary to an intuition educated only by Euclidean geometry. 

One has the following characterizations of ultrametrics: 

\begin{proposition}   \label{prop:equivultra}
     Let $(M,d)$ be a metric space. Then the following properties are equivalent:
        \begin{enumerate}
             \item  $(M,d)$ is ultrametric.
             \item  The triangles are all isosceles with two equal sides not less than the third side. 
             \item All the points of a closed ball are centers of it. 
             \item  Two closed balls are either disjoint, or one is included in the other. 
        \end{enumerate}
\end{proposition}

\begin{proof}
    All the equivalences are elementary but instructive to check. We leave their proofs as 
    exercises. 
\end{proof}

\begin{example}  \label{ex:fourpointum}
     Consider a set $M = \{A, B, C, D\}$ and a distance function $d$ on it such that:
        $d(B, C) =1, d(A,B) = d(A,C) = 2, d(A,D) = d(B, D) = d(C, D) = 5$. Note that one may 
        embed $(M, d)$ isometrically into a $3$-dimensional Euclidean space by choosing 
        an isosceles triangle $ABC$ with the given edge lengths, and by choosing then 
        the point $D$ on the perpendicular to the plane of the triangle passing through 
        its circumcenter. Let us look for the closed balls of this finite metric space. 
        For radii less than $1$, they are singletons. For radii in the interval 
        $[1, 2)$, we get the sets $\{B,C\}, \{A\}, \{D\}$. Note that both $B$ and $C$ are centers 
        of the ball $\{B,C\}$, that is, $\cB(B, r) = \cB(C, r) = \{B,C\}$ for every $r \in [1,2)$. 
        Once the radius belongs to the interval $[2, 5)$, the balls are $\{A, B, C\}$ and $\{D\}$. 
        Finally, for every radius $r \in [5, \infty)$, there is only one closed ball, the whole set. 
        Figure \ref{fig:ballsfour} depicts the set $\{A, B, C, D\}$ as well as the mutual 
        distances and the associated set of closed balls.
\end{example}

\begin{figure}[h!]
    \begin{center}
\begin{tikzpicture}[scale=0.4]

    \node[draw,circle, inner sep=2pt,color=black, fill=black] at (6.5,-1){};
     \node[draw,circle, inner sep=2pt,color=black, fill=black] at (9.5,-1){};
     \node[draw,circle, inner sep=2pt,color=black, fill=black] at (8, 3){};
     \node[draw,circle, inner sep=2pt,color=black, fill=black] at (15, 1){};
    
     \node[draw,circle, inner sep=5pt,color=black] at (6.5,-1){};
      \node[draw,circle, inner sep=5pt,color=black] at (9.5,-1){};
      \node[draw,circle, inner sep=5pt,color=black] at (8, 3){};
       \node[draw,circle, inner sep=5pt,color=black] at (15, 1){};
      
    \node[draw,circle, inner sep=25pt,color=black] at (8,-1){};
     \node[draw,circle, inner sep=45pt,color=black] at (8,0){};
    \node[draw,circle, inner sep=65pt,color=black] at (10,0){};
    
    \draw [-, color=black, dashed](6.5,-1) -- (9.5,-1) -- (15,1) -- (8,3);
     \draw [-, color=black, dashed](15, 1) -- (6.5,-1) -- (8,3) -- (9.5, -1);
     
     \node [left] at (7.5,3) {$A$};
      \node [below] at (6.5, -1.5) {$B$};
       \node [below] at (9.5, -1.5) {$C$};
       \node [below] at (15, 0.5) {$D$};
       
        \node [below] at (8,-1) {$1$};
      \node [left] at (7.2, 0.8) {$2$};
       \node [right] at (8.8, 0.8) {$2$};
       \node [below] at (12.5, 0.2) {$5$};
        \node [below] at (11.7, 1.4) {$5$};
       \node [below] at (12.3, 2.9) {$5$};    
    
    \end{tikzpicture}
\end{center}
\caption{The balls of an ultrametric space with four points} 
 \label{fig:ballsfour}
   \end{figure}
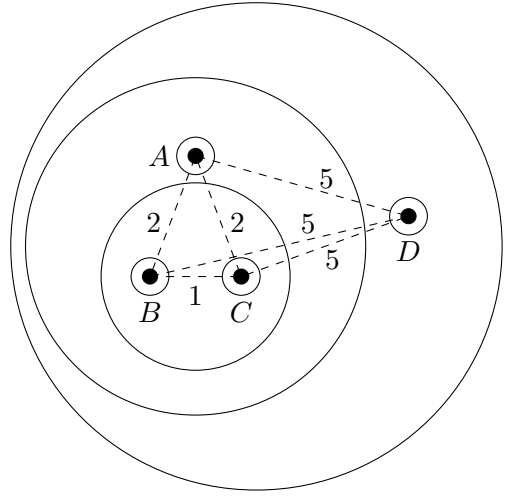

Example \ref{ex:fourpointum}  illustrates the fact that neither the center nor the radius of a closed ball 
of a finite ultrametric space is well-defined, once the ball has more than one element. 
Instead, every closed ball has a well-defined \emph{diameter}:

\begin{definition}  \label{def:diam}   \index{diameter!of a closed ball} 
  The {\bf diameter} of a closed ball in a finite metric space is the maximal distance between pairs of 
  not necessarily distinct points of it.  
\end{definition}

The last characterization of ultrametrics 
in Proposition \ref{prop:equivultra}  shows that the set $\boxed{\mathrm{Balls}(M,d)}$  
of closed balls of an 
ultrametric space $(M,d)$ is a \emph{hierarchy} on $M$, in the following sense:

\begin{definition}  \label{def:hier}  \index{hierarchy!on a set} 
     A {\bf hierarchy} on a set $M$ is a subset $\cH$ of its power set $\cP(M)$, 
     satisfying the following properties:
        \begin{itemize}
            \item   $\emptyset \notin \cH$. 
            \item  The singletons belong to $\cH$. 
            \item  $M$ belongs to $\cH$. 
            \item Two elements of $\cH$ are either disjoint, or one is included into the other. 
        \end{itemize}
\end{definition}

If $\cH$ is a hierarchy on a set $M$, it may be endowed with the inclusion partial order. 
We will consider instead its reverse partial order $\boxed{\preceq_{\cH}}$, defined by:
  \[ A \preceq_{\cH} B \:  \Longleftrightarrow \:   A \supseteq B, \:   \mbox{ for all } A, B \in \cH. \]
 Reversing the inclusion partial order has the advantage of identifying the leaves of the 
 corresponding rooted tree with the maximal elements of the poset $(\cH, \preceq_{\cH})$ 
 (see Proposition \ref{prop:treehier} below).

When $M$ is \emph{finite}, one may represent the poset $(\cH, \preceq_{\cH})$ 
using its associated {\em Hasse diagram}: 

\begin{definition} \label{def:hassediag} 
    Let $(X, \preceq)$ be a finite poset. Its {\bf Hasse diagram}  \index{Hasse diagram} 
    is the directed graph whose set of vertices is $X$, two vertices $a,b \in X$ 
    being joined by an edge  oriented from $a$ to $b$ 
    whenever $a \prec b$ and the two points are {\bf directly comparable}, 
    \index{direct comparability} 
    that is,  there is no other element of $X$ lying strictly between them. 
\end{definition}

Hasse diagrams of finite posets are abstract oriented acyclic graphs. This means that they have 
no directed cycles, which is a consequence of the fact that a partial order is antisymmetric and 
transitive. Hasse diagrams are not necessarily planar, but, as all finite graphs, they 
may be always immersed in the plane in such a way that any pair of edges intersect transversely.  
When drawing a Hasse diagram in the plane as an immersion, 
    we will use the convention to place the vertex $a$ of the Hasse diagram  
    \emph{below} the vertex $b$ whenever $a \prec b$. This is always 
    possible because of the absence of directed cycles. This convention makes unnecessary 
    adding arrowheads along the edges in order to indicate their orientations. 

\begin{example}  \label{ex:div12}
   Consider the finite set $\{1, 2, 3, 4, 6, 12\}$ of positive divisors of $12$, partially ordered by divisibility: 
   $a \preceq b$ if and only if $a$ divides $b$. Its Hasse diagram is drawn in 
   Figure \ref{fig:Hassediag12}. 
\end{example}

 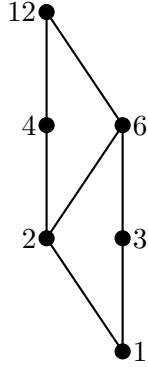
\begin{figure}[h!]
    \begin{center}
\begin{tikzpicture}[scale=0.5]

    \node[draw,circle, inner sep=2pt,color=black, fill=black] at (0,10){};
     \node[draw,circle, inner sep=2pt,color=black, fill=black] at (0,7){};
     \node[draw,circle, inner sep=2pt,color=black, fill=black] at (0, 4){};
     \node[draw,circle, inner sep=2pt,color=black, fill=black] at (2, 7){};
    \node[draw,circle, inner sep=2pt,color=black, fill=black] at (2, 4){};
     \node[draw,circle, inner sep=2pt,color=black, fill=black] at (2, 1){};
         
    \draw [-, color=black, thick](0,10) -- (0,4) -- (2, 1) -- (2,7) -- (0,10);
     \draw [-, color=black, thick](2, 7) -- (0,4) ;
         
     \node [left] at (0,10) {$12$};
      \node [left] at (0, 7) {$4$};
       \node [left] at (0, 4) {$2$};
       \node [right] at (2, 7) {$6$};
       \node [right] at (2, 4) {$3$};   
      \node [right] at (2, 1) {$1$};
    
    \end{tikzpicture}
\end{center}
\caption{The Hasse diagram of the set of positive divisors of $12$.} 
 \label{fig:Hassediag12}
   \end{figure}

The Hasse diagrams of finite hierarchies are special kinds of graphs:

\begin{proposition}  \label{prop:treehier}
  The Hasse diagram of a hierarchy $(\cH, \preceq_H)$ on a finite set $M$ is a tree in which the maximal 
  directed paths start from $M$ and terminate at the singletons. Moreover, for each vertex 
  which is not a singleton, there are at least two edges starting from it. 
\end{proposition}

\begin{proof} We sketch a proof, leaving the details to the reader. 

    The first statement results from the fact that the singletons of $M$ are exactly the 
    maximal elements of the poset $(\cH, \preceq_H)$, that $M$ itself is the unique minimal element 
    and that all the elements of a hierarchy which contain a given element are 
    totally ordered by inclusion. 
    
    Let us prove the second statement. Consider $B_1\in \cH$ and assume that it is not a 
    singleton. This means that it is not minimal for inclusion, therefore there exists $B_2 \in \cH$ 
    such that $B_2 \subsetneq B_1$ and $B_2$ is directly comparable to $B_1$. Let $A$ 
    be a point of $B_1 \: \setminus B_2$. Consider  $B_3 \in \cH$ which contains the 
    point $A$, is included into 
    $\cB_1$ and is directly comparable to it. As $A \in B_3 \: \setminus \: B_2$, this shows that 
    $B_3$ is not included in $B_2$. We want to show that the two sets $B_2$ and $B_3$ 
    are disjoint. Otherwise, 
    by the definition of a hierarchy, we would have $B_2 \subsetneq B_3 \subsetneq B_1$, 
    which contradicts the assumption that $B_1$ and $B_2$ are directly comparable. 
\end{proof}

\begin{example}  \label{ex:hassediaghier}
     Consider the ultrametric space of Example \ref{ex:div12}, 
     represented in Figure \ref{fig:ballsfour}. We repeat it on the left 
     of Figure \ref{fig:treeballs}. The Hasse diagram of the 
     hierarchy of its closed balls is drawn on the right of Figure \ref{fig:treeballs}. Near each vertex 
     is represented the diameter of the corresponding ball. We have added a 
     \emph{root vertex}, connected to the vertex representing the whole set. 
     It may be thought as a larger ball, obtained by adding formally to $M= \{A, B, C, D\}$ 
     a point  $\omega$, infinitely distant from each point of $M$. This larger ball 
     is the set $\overline{M} := M \cup \{\omega\}$.      
\end{example}

    \begin{figure}[h!]
    \begin{center}
\begin{tikzpicture}[scale=0.5]

\begin{scope}[shift={(-25,6)}]

 \node[draw,circle, inner sep=2pt,color=black, fill=black] at (6.5,-1){};
     \node[draw,circle, inner sep=2pt,color=black, fill=black] at (9.5,-1){};
     \node[draw,circle, inner sep=2pt,color=black, fill=black] at (8, 3){};
     \node[draw,circle, inner sep=2pt,color=black, fill=black] at (15, 1){};
    
     \node[draw,circle, inner sep=5pt,color=black] at (6.5,-1){};
      \node[draw,circle, inner sep=5pt,color=black] at (9.5,-1){};
      \node[draw,circle, inner sep=5pt,color=black] at (8, 3){};
       \node[draw,circle, inner sep=5pt,color=black] at (15, 1){};
      
    \node[draw,circle, inner sep=25pt,color=black] at (8,-1){};
     \node[draw,circle, inner sep=45pt,color=black] at (8,0){};
    \node[draw,circle, inner sep=65pt,color=black] at (10,0){};
    
    \draw [-, color=black, dashed](6.5,-1) -- (9.5,-1) -- (15,1) -- (8,3);
     \draw [-, color=black, dashed](15, 1) -- (6.5,-1) -- (8,3) -- (9.5, -1);
     
     \node [left] at (7.5,3) {$A$};
      \node [below] at (6.5, -1.5) {$B$};
       \node [below] at (9.5, -1.5) {$C$};
       \node [below] at (15, 0.5) {$D$};
       
        \node [below] at (8,-1) {$1$};
      \node [left] at (7.2, 0.8) {$2$};
       \node [right] at (8.8, 0.8) {$2$};
       \node [below] at (12.8, 0.2) {$5$};
        \node [below] at (11.7, 1.4) {$5$};
       \node [below] at (12.3, 2.9) {$5$};    

\end{scope}

 \draw [->, color=blue, thick] (-6, 6) -- (-4, 6) ;

    \node[draw,circle, inner sep=2pt,color=black, fill=black] at (0,10){};
     \node[draw,circle, inner sep=2pt,color=black, fill=black] at (2,10){};
     \node[draw,circle, inner sep=2pt,color=black, fill=black] at (4, 10){};
     \node[draw,circle, inner sep=2pt,color=black, fill=black] at (8, 10){};

    \draw [-, color=black, thick](0,10) -- (3,4) -- (8, 10);
     \draw [-, color=black, thick](2, 10) -- (3,8) -- (4, 10);
     \draw [-, color=black, thick](3,8) -- (2,6);
     \draw [-, color=black, thick](3, 4) -- (3,2);
     
      \draw [-, color=black, very thick](2, 2) -- (4,2);
      
      \node[draw,circle, inner sep=2pt,color=black, fill=white] at (3,8){};
      \node[draw,circle, inner sep=2pt,color=black, fill=white] at (2,6){};
      \node[draw,circle, inner sep=2pt,color=black, fill=white] at (3, 4){};
     
     \node [above] at (0,10.5) {$\{A\}$};
      \node [above] at (2, 10.5) {$\{B\}$};
       \node [above] at (4, 10.5) {$\{C\}$};
       \node [above] at (8, 10.5) {$\{D\}$};
       
        \node [right] at (3.2, 8) {$\{B, C\}$};
        \node [left] at (1.8, 6) {$\{A, B, C\}$};  
         \node [right] at (3.2, 4) {$\{A, B, C, D\}$};
          \node [right] at (3.2, 2.5) {$\{A, B, C, D, \omega\}$};
       
        \node [left] at (0,10) {$0$};
      \node [left] at (2, 10) {$0$};
       \node [left] at (4, 10) {$0$};
       \node [left] at (8, 10) {$0$};
       
        \node [left] at (3, 8) {$1$};
       \node [right] at (2, 6) {$2$};   
        \node [left] at (3, 4) {$5$};
       \node [above] at (2.5, 2) {$\infty$};    
    
    \end{tikzpicture}
\end{center}
\caption{The tree of the hierarchy of closed balls of Example \ref{ex:hassediaghier}} 
 \label{fig:treeballs}
   \end{figure}
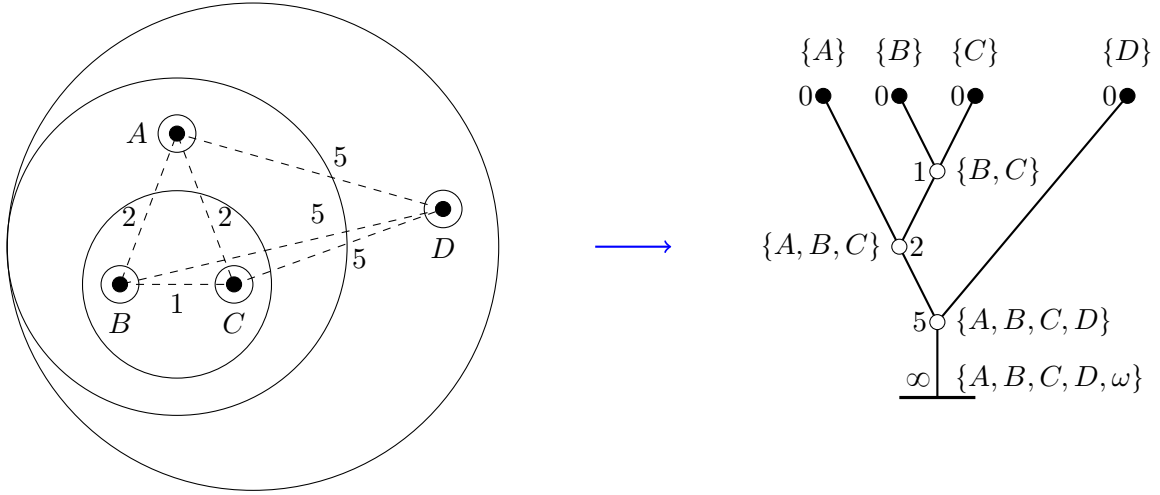

    One may formalize in the following way the construction performed in Example \ref{ex:hassediaghier}:
    
    \begin{definition}   \label{def:treehierarchy}  \index{tree!of a hierarchy} 
         The {\bf tree} of a hierarchy $(\cH, \preceq_H)$ on a finite poset $M$ is its Hasse diagram,  
         completed with a root  representing the set $\boxed{\overline{M}} := M \cup \{\omega\}$, 
         joined with the vertex representing $M$ and rooted at $\overline{M}$. Here $\omega$ 
         is a point distinct from the points of $M$. 
    \end{definition}
    
    The tree of a hierarchy is a \emph{rooted tree} in the following sense:
    
    \begin{definition}  \label{def:rootree}  \index{tree!rooted}
    A {\bf rooted tree} is a tree with a distinguished vertex, called its {\bf root}.   \index{root!of a tree}
    If $\Theta$ is a rooted tree with root $r$, then the vertex set of $\Theta$ gets 
    partially ordered by declaring that $\boxed{a \preceq_r b}$ if and only if the unique segment 
    $[r, a]$ joining $r$ to $a$ in the tree is contained in $[r,b]$. 
    \end{definition}
    
    When $\Theta$ is 
    the rooted tree of a hierarchy $\cH$ on a finite set $M$, then the partial order 
    $\preceq_{\overline{M}}$ defined by choosing $\overline{M}$ as root 
    restricts to the partial order $\preceq_{\cH}$ if one identifies 
    the set $\cH$ with the set of vertices of $\Theta$ which are distinct from the root. 
    
   Proposition \ref{prop:treehier} may be reformulated in the following way 
   as a list of properties of the tree of the hierarchy:
   
   \begin{proposition} \label{prop:reformtree}
          Let $\Theta$ be the tree of a hierarchy on a finite set, and let $r$ be its root. 
          Then $r$ is a vertex of valency $1$ and there are no vertices of valency $2$. 
   \end{proposition} 
   
   This proposition motivates the following definition:
   
   \begin{definition}  \label{def:hiertree}   \index{tree!hierarchical}
      A rooted tree whose root is of valency $1$ and which does not possess vertices 
      of valency $2$ is a {\bf hierarchical tree}. The {\bf hierarchy} of a hierarchical tree $(\Theta, r)$ 
      is constructed in the following way: \index{hierarchy!of a hierarchical tree}
      \begin{itemize}
          \item  Define $M$ to be the set of {\bf leaves} of the rooted tree $(\Theta, r)$, that is, the set 
             of vertices of valency $1$ which are distinct from the root $r$. \index{leaf!of a tree}
          \item For each vertex $p$ of $\Theta$ different from the root, consider the subset of $M$ 
              consisting of the leaves $a$ such that $p \preceq_r a$. 
      \end{itemize}
   \end{definition}
      
      We leave as an exercise to prove:
      
      \begin{proposition}  \label{prop:equivconstr}
             The constructions of Definitions \ref{def:treehierarchy} and \ref{def:hiertree}, 
             which associate a hierarchical tree 
             to a hierarchy on a finite set and a hierarchy to a hierarchical tree are inverse of 
             each other. 
      \end{proposition}
      
      As a preliminary to the proof, one may test the truth of the proposition on the 
      example of Figure \ref{fig:treeballs}. 
      
      \medskip
      Let us return to finite ultrametric spaces $(M, d)$. We saw that the set 
      $\mathrm{Balls}(M,d)$ of its closed balls is a hierarchy on $M$. Proposition \ref{prop:equivconstr} 
      shows that one may think about this hierarchy as a special kind of rooted tree, 
      namely, a hierarchical tree. This hierarchical tree alone 
      does not allow to get back the distance function $d$. How to encode it on the tree? 
    
     The idea is to look at the function defined on $\mathrm{Balls}(M,d)$, 
     which associates to each ball its diameter 
     (see Definition \ref{def:diam}): 

\begin{proposition}  \label{prop:hierultr}
      Let $(M,d)$ be a finite ultrametric space. Then the map which sends each closed ball 
      to its diameter is a strictly decreasing $[0, \infty)$-valued 
      function defined on the poset $(\mathrm{Balls}(M,d), \preceq)$, taking the value $0$ 
      exactly on the singletons of $M$. Equivalently, 
      it is a strictly decreasing $[0, \infty]$-valued function on the set of vertices of the 
      tree of the hierarchy, vanishing on the set $M$ of leaves and taking 
      the value $\infty$ on the root. 
\end{proposition}

As an example, one may look again at Figure \ref{fig:treeballs}. The value taken by the 
previous diameter function is written near each vertex of the hierarchical tree. 

If $(\Theta, r)$ is a hierarchical tree, denote by $\boxed{V(\Theta)}$ its set of vertices and by 
$\boxed{a \wedge_r b}$ the infimum of $a$ and $b$ relative to $\preceq_r$, whenever 
$a,b \in V(\Theta)$. This infimum may be characterized by the property that 
$[r, a] \cap [r,b] = [r, a \wedge_r b]$.  The following is a converse of Proposition \ref{prop:hierultr}: 

\begin{proposition}  \label{prop:convultram}
     Let $(\Theta, r)$ be a hierarchical tree and $\lambda : V(\Theta) \to [0, \infty]$ 
     be a strictly decreasing function relative to the partial 
     order $\preceq_r$ on $\Theta$ induced by the root. Assume that $\lambda$ 
     vanishes on the set $M$ of leaves of 
     $\Theta$ and takes the value $\infty$ at $r$. Then the map 
          $$\begin{array}{lccc}
                   d: &   M \times M & \to & [0, \infty) \\
                       &   (a,b)                   & \to &  \lambda(a \wedge_r b) 
              \end{array}$$
      is an ultrametric on $M$. 
\end{proposition}

 Let us introduce a special name for the functions appearing in Proposition \ref{prop:convultram}:

\begin{definition}   \label{def:depthfunct}  \index{depth function}
     Let $(\Theta, r)$ be a hierarchical tree. A {\bf depth function} on it is a function 
     $\lambda : V(\Theta) \to [0, \infty]$ which satisfies the following properties: 
     \begin{itemize}
       \item it is strictly decreasing relative to the partial 
            order $\preceq_r$ on $\Theta$ induced by the root $r$; 
        \item it vanishes on the set of leaves of $\Theta$; 
        \item it takes the value $\infty$ at the root $r$. 
     \end{itemize}
\end{definition}

Note that the first two conditions of Definition \ref{def:depthfunct} imply 
 that a depth function vanishes \emph{exactly} on the set of leaves of the 
underlying hierarchical tree.

One has the following analog of Proposition \ref{prop:equivconstr}:

\begin{proposition}  \label{prop:equivultramtree}
   The constructions of Propositions \ref{prop:hierultr} and \ref{prop:convultram} are inverse of each other. 
   That is, giving an ultrametric on a finite set $M$ is equivalent to giving a depth function on a 
   hierarchical tree whose set of leaves is $M$.  
\end{proposition}

It is this proposition which allows to think about an ultrametric as a special kind of rooted and 
decorated tree. We leave its proof as an exercise (see \cite{BD 98}).

\section{A proof of P\l oski's theorem using Eggers-Wall trees}
\label{sec:proofthm}

In this section we sketch a proof of P\l oski's theorem \ref{thm:ploski} using the equivalence 
between ultrametrics on finite sets and certain kinds of rooted trees formulated in 
Proposition \ref{prop:equivultramtree}. The rooted trees used in this proof are the 
\emph{Eggers-Wall trees} of a plane curve singularity relative to smooth reference 
branches. The precise definition of Eggers-Wall trees is not given, because the proofs 
of the subsequent generalizations of P\l oski's theorem will be of a completely different spirit. 
\medskip

Instead of working both with multiplicities and intersection numbers as in P\l oski's 
original statement, we will work only with the latest ones. 

Let $S$ be a smooth germ of surface and $L \hookrightarrow S$ be a \emph{smooth branch}. Define 
the following function on the set of branches on $S$ which are different from $L$:

\begin{equation}  \label{eq:ultramrelbranch}
    \begin{array}{lccc}
                   \boxed{u_L} : &   \left(\cB(S) \setminus \{L \} \right)^2 & \to & \R_+ \\
                       &   (A,B)   & \to &     
                             \left\{ \begin{array}{cc}
                                        \dfrac{(L \cdot A) \cdot (L \cdot B)}{A \cdot B} & \mbox{ if } A \neq B, \\
                                           0 & \mbox{ otherwise}.
                                      \end{array} \right.
              \end{array}
\end{equation}

In the remaining part of this section  we will sketch a proof of:

\begin{theorem}   \label{thm:ulultram}
    The function $u_L$ is an ultrametric. 
\end{theorem}

We leave as an exercise to show using Proposition \ref{prop:ineqint}  
that Theorem \ref{thm:ulultram} implies the reformulation given in Theorem 
\ref{thm:ploskibis} of P\l oski's Theorem \ref{thm:ploski}. 

Our proof of Theorem \ref{thm:ulultram} will pass through the notion of 
\emph{Eggers-Wall tree} \index{tree!Eggers-Wall} associated to a plane curve singularity relative to 
a smooth branch of reference $L$ (see the proof of Theorem \ref{thm:idtrees} below). 
Let us illustrate it by an example.

  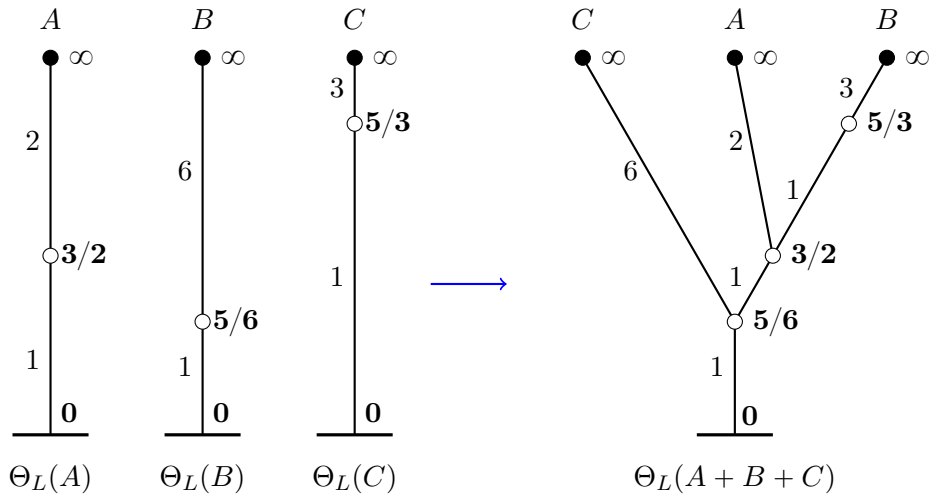
\begin{figure}[h!]
    \begin{center}
\begin{tikzpicture}[scale=0.5]

    \draw [-, color=black, thick](-14,0) -- (-14,10) ;
     \node[draw,circle, inner sep=2pt,color=black, fill=black] at (-14,10){};
     \node[draw,circle, inner sep=2pt,color=black, fill=white] at (-14,  4.75){};
     \node [right] at (-14, 4.75) {$\mathbf{3/2}$};
     \node [left] at (-14, 2) {$1$};
     \node [left] at (-14, 7.8) {$2$};
      \node [above] at (-14, 10.5) {$A$};
       \node [right] at (-13.8, 10) {$\mathbf{\infty}$};
       \node [right] at (-14, 0.6) {$\mathbf{0}$};  
       \draw [-, color=black, very thick](-15, 0) -- (-13,0);
       \node [below] at (-14, -0.5) {$\Theta_L(A)$};  

    \draw [-, color=black, thick](-10,0) -- (-10,10) ;
     \node[draw,circle, inner sep=2pt,color=black, fill=black] at (-10,10){};
     \node[draw,circle, inner sep=2pt,color=black, fill=white] at (-10,  3){};
     \node [right] at (-10, 3) {$\mathbf{5/6}$};
     \node [left] at (-10, 1.8) {$1$};
     \node [left] at (-10, 7) {$6$};
      \node [above] at (-10, 10.5) {$B$};
      \node [right] at (-9.8, 10) {$\mathbf{\infty}$};
      \node [right] at (-10, 0.6) {$\mathbf{0}$};  
       \draw [-, color=black, very thick](-11, 0) -- (-9,0);
       \node [below] at (-10, -0.5) {$\Theta_L(B)$};  

    \draw [-, color=black, thick](-6,0) -- (-6,10) ;
      \node[draw,circle, inner sep=2pt,color=black, fill=black] at (-6,10){};
      \node[draw,circle, inner sep=2pt,color=black, fill=white] at (-6,  8.25){};
      \node [right] at (-6, 8.25) {$\mathbf{5/3}$};
      \node [left] at (-6, 4.2) {$1$};
     \node [left] at (-6, 9.2) {$3$};
       \node [above] at (-6, 10.5) {$C$};
       \node [right] at (-5.8, 10) {$\mathbf{\infty}$};
       \node [right] at (-6, 0.6) {$\mathbf{0}$};  
        \draw [-, color=black, very thick](-7, 0) -- (-5,0);
        \node [below] at (-6, -0.5) {$\Theta_L(C)$};


     \draw [->, color=blue, thick] (-4, 4) -- (-2, 4) ; 

    \node[draw,circle, inner sep=2pt,color=black, fill=black] at (0,10){};
     \node[draw,circle, inner sep=2pt,color=black, fill=black] at (4,10){};
     \node[draw,circle, inner sep=2pt,color=black, fill=black] at (8, 10){};

    \draw [-, color=black, thick](4,0) -- (4,3) ;
     \draw [-, color=black, thick](4, 3) --  (0, 10);
     \draw [-, color=black, thick](5, 4.75) -- (4,10);
     \draw [-, color=black, thick](4, 3) -- (8,10);
     
      \draw [-, color=black, very thick](3, 0) -- (5,0);
      
      \node[draw,circle, inner sep=2pt,color=black, fill=white] at (4,3){};
      \node[draw,circle, inner sep=2pt,color=black, fill=white] at (5,  4.75){};
      \node[draw,circle, inner sep=2pt,color=black, fill=white] at (7, 8.25){};
     
     \node [above] at (0,10.5) {$C$};
      \node [above] at (4, 10.5) {$A$};
       \node [above] at (8, 10.5) {$B$};
       \node [right] at (0.2, 10) {$\mathbf{\infty}$};
       \node [right] at (4.2, 10) {$\mathbf{\infty}$};
       \node [right] at (8.2, 10) {$\mathbf{\infty}$};
           
        \node [right] at (7.2, 8.25) {$\mathbf{5/3}$};
         \node [right] at (5.2, 4.75) {$\mathbf{3/2}$};
          \node [right] at (4.2, 3) {$\mathbf{5/6}$};
           \node [above] at (4.4, 0) {$\mathbf{0}$};

        \node [left] at (4, 1.8) {$1$};
         \node [left] at (4.5, 4.2) {$1$};
          \node [left] at (6, 6.5) {$1$};
       \node [right] at (0.8, 7) {$6$};   
        \node [left] at (4.5, 7.8) {$2$};
         \node [left] at (7.4, 9.2) {$3$};
         
         \node [below] at (4, -0.5) {$\Theta_L(A+B+C)$};

    \end{tikzpicture}
\end{center}
\caption{The Eggers-Wall tree of the plane curve singularity of Example \ref{ex:basicbis}} 
 \label{fig:EWtree}
   \end{figure}

   \begin{example}  \label{ex:basicbis}
    Consider again the branches $A= Z(y^2 - x^3), B= Z(y^3 - x^5), C= Z(y^6 - x^5)$ 
    on $S = (\C^2, 0)$ of Example \ref{ex:basic}. 
    Assume that the branch $L$ is the germ at $0$ of the $y$-axis $Z(x)$. 
    The defining equations of the three branches $A, B, C$ may be considered as polynomial 
    equations in the variable $y$. As such, they admit the following roots which are fractional 
    powers of $x$:
      $$\begin{array}{l}
            A : \:    x^{3/2},  \\
            B : \:   x^{5/3}, \\
            C : \:   x^{5/6}. 
      \end{array}$$   
     Associate to the root $x^{3/2}$ a  compact segment $\Theta_L(A)$ identified with 
     the interval $[0, \infty]$ using an \emph{exponent function} $e_L : \Theta_L(A) \to [0, \infty]$ 
     and mark on it the point $e_L^{-1}(3/2)$ with exponent $3/2$. Define also an  
     \emph{index function} $i_L : \Theta_L(A) \to \N^*$, constantly equal to $1$ on the 
     interval $[e_L^{-1}(0), e_L^{-1}(3/2)]$ and to $2$ on the interval 
     $(e_L^{-1}(3/2), e_L^{-1}(\infty)]$ (see the left-most segment of Figure \ref{fig:EWtree}). 
     Here the number $2$ is to be thought as the minimal positive 
     denominator of the exponent $3/2$ of the monomial $x^{3/2}$. 
     The segment $\Theta_L(A)$ endowed with the two functions $e_L$ and $i_L$ is 
     the \emph{Eggers-Wall tree} of the branch $A$ relative to the branch $L$. 
     It is considered as a rooted tree with root $e_L^{-1}(0)$, labeled with the 
     branch $L$. Its leaf $e_L^{-1}(\infty)$ is labeled with the branch $A$. Consider analogously  
     the Eggers-Wall trees $\Theta_L(B)$ and $\Theta_L(C)$, endowed with pairs of exponent and 
     index functions and labeled roots and leaves (see the left part of Figure \ref{fig:EWtree}). 
     
     Look now at the plane curve singularity $A + B + C$. Its Eggers-Wall tree 
     $\Theta_L(A + B + C)$ relative to the branch $L$ is obtained from the individual 
     trees $\Theta_L(A), \Theta_L(B), \Theta_L(C)$ by a gluing process, which identifies 
     two by two initial segments of those trees. 
     
     Consider for instance the segments 
     $\Theta_L(A), \Theta_L(B)$. Look at the order of the difference 
     $x^{3/2} - x^{5/3}$ of the roots which generated them, seen as a series with 
     fractional exponents. This order is the fraction $3/2$, 
     because $3/2 < 5/3$. 
     Identify then the points with the same exponent $\leq 3/2$ of the segments 
     $\Theta_L(A), \Theta_L(B)$. One gets a rooted tree $\Theta_L(A + B)$ 
     with root labeled by $L$ and 
     with two leaves, labeled by the branches $A, B$. The exponent and index functions 
     of the trees $\Theta_L(A), \Theta_L(B)$ descend to functions with the same name 
     $e_L, i_L$  defined on $\Theta_L(A + B)$. Endowed with those functions, 
     $\Theta_L(A + B)$ is the \emph{Eggers-Wall tree} of the curve singularity $A + B$. 
     
     If one considers now the curve singularity $A + B + C$, then one glues analogously 
     the three pairs of trees obtained from $\Theta_L(A), \Theta_L(B), \Theta_L(C)$. 
     The resulting Eggers-Wall tree $\Theta_L(A + B + C)$ is drawn on the right side of Figure 
     \ref{fig:EWtree}. It is also endowed with two functions $e_L, i_L$, obtained by 
     gluing the exponent and index functions of the trees $\Theta_L(A), \Theta_L(B), \Theta_L(C)$. 
     Its marked points are its ends, its bifurcation points and the images 
     of the discontinuity points of the index function of the Eggers-Wall tree of each branch. 
     Near each marked point is written the corresponding value of the exponent function. 
     The index function is constant on each segment $(a,b]$ joining two marked points 
     $a$ and $b$, where $a \prec_L b$. Here $\preceq_L$ denotes the partial order on the tree 
      $\Theta_L(A + B + C)$ determined by the root $L$ (see Definition \ref{def:rootree}).     
\end{example}

One may associate analogously an {\bf Eggers-Wall tree} \index{tree!Eggers-Wall} 
$\boxed{\Theta_L(D)}$ to any plane 
curve singularity $D$, relative to a \emph{smooth} reference branch $L$. It is a rooted tree 
endowed with an {\bf exponent function} \index{function!exponent} 
$\boxed{e_L} : \Theta_L(D) \to [0, \infty]$ and 
 an {\bf index function} \index{function!index} 
 $\boxed{i_L} : \Theta_L(D) \to \N^*$. The tree and both functions 
 are constructed using Newton-Puiseux series expansions of the roots of a Weierstrass 
 polynomial $f \in \C[[x]][y]$ defining $D$ in a coordinate system $(x,y)$ 
 such that $L = Z(x)$. The triple $(\Theta_L(D), e_L, i_L)$ is independent of the choices 
  involved in the previous definition (see \cite[Proposition 103]{GBGPPP 18}). One may 
  find the precise definition and examples of Eggers-Wall trees in Section 4.3 of the previous 
  reference and in \cite[Sect. 3]{GBGPPP 19}. Historical remarks about this notion 
  may be found in \cite[Rem. 3.18]{GBGPPP 19} and \cite[Sect. 6.2]{GBGPPP 20}. 
  The name, introduced in author's thesis \cite{PP 01}, 
  makes reference to Eggers' 1982 paper \cite{E 82} and to Wall's 2003 paper \cite{W 03}.

   What allows us to prove Theorem \ref{thm:ulultram} using Eggers-Wall trees is that 
the values $u_L(A,B)$ of the function $u_L$ defined by relation (\ref{eq:ultramrelbranch}) 
are determined 
in the following way from the Eggers-Wall tree $\Theta_L(D)$, for each pair of distinct 
branches $(A, B)$ of $D$ (recall from the paragraph preceding Proposition \ref{prop:convultram} 
that $A \wedge_L B$ denotes the infimum of $A$ and $B$ 
relative to the partial order $\preceq_L$ induced by the root $L$ of $\Theta_L(D)$):

\begin{theorem}   \label{thm:intmin}
     For each pair $(A, B)$ of distinct branches of $D$ and every smooth 
     reference branch $L$ different from the branches of $D$, one has:
     $$\dfrac{1}{u_L(A, B)} = \int_L^{A \wedge_L B} \dfrac{de_L}{i_L}. $$
\end{theorem}

\begin{example}  \label{ex:verifint}
   Let us verify the equality stated in Theorem \ref{thm:intmin} on the branches of 
   Example \ref{ex:basicbis}. Looking at the Eggers-Wall tree $\Theta_L(A + B + C)$ on the 
   right side of Figure \ref{fig:EWtree}, we see that:
      \[ \int_L^{A \wedge_L B} \dfrac{de_L}{i_L} = \int_0^{3/2} \dfrac{de}{1} = \dfrac{3}{2}. \]
   But $1/ u_L(A, B) = (A \cdot B) / \left( (L \cdot A) (L \cdot B) \right) = 
   (A \cdot B) / \left( m(A) \cdot m(B)\right) = 3/2$, as was computed in Example \ref{ex:basic}. 
    The equality is verified. We have used the fact that both $A$ and $B$ are transversal 
    to $L$, which implies that $L \cdot A = m(A)$ and $L \cdot B = m(B)$. 
\end{example}

In equivalent formulations which use so-called \emph{characteristic exponents}, 
Theorem \ref{thm:intmin} goes back 
to Smith \cite[Section 8]{S 75}, Stolz \cite[Section 9]{S 79} and Max Noether \cite{N 90}. 
A modern proof, based on Proposition \ref{prop:intpar}, may be found in \cite[Thm. 4.1.6]{W 04}.

As a consequence of Theorem \ref{thm:intmin}, we get the following strengthening of 
Theorem \ref{thm:ulultram}:

\begin{theorem}  \label{thm:idtrees}
      Let $D$ be a plane curve singularity. Denote by $\cF(D)$ the set of branches of $D$. 
      Let $L$ be a reference smooth branch which does not belong to $\cF(D)$. 
      Then the function $u_L$ is an ultrametric in restriction to $\cF(D)$ and its associated 
      rooted tree is isomorphic as a rooted tree with labeled leaves to the 
      Eggers-Wall tree $\Theta_L(D)$. 
\end{theorem}

\begin{proof}
     Consider $\Theta_L(D)$ as a topological tree with vertex set equal to its set of 
     ends and of ramification points. Root it at $L$. Then it becomes a hierarchical tree 
     in the sense of Definition \ref{def:hiertree}. The function 
         $$ P \to \left( \int_L^{P} \dfrac{de_L}{i_L}\right)^{-1}$$
     is a depth function on it, in the sense of Definition \ref{def:depthfunct}. 
     Using Theorem \ref{thm:intmin} and Proposition \ref{prop:equivultramtree}, we get 
     Theorem \ref{thm:idtrees}.      
\end{proof}

For more details about the proof of P\l oski's theorem presented in this section, see 
\cite[Sect. 4.3]{GBGPPP 18}.

\section{An ultrametric characterization of arborescent singularities}
\label{sec:ultramchar}

In this section we state a generalization of Theorem \ref{thm:ulultram} 
for all \emph{arborescent singularities} and the fact that it characterizes 
this class of normal surface singularities. We start by recalling the needed 
notions of \emph{embedded resolution} and associated \emph{dual graph} of a finite 
set of branches contained in a normal surface singularity. 
\medskip

From now on, $S$ denotes an arbitrary {\bf normal surface singularity}, that is, 
\index{singularity!normal surface} 
a germ of normal complex analytic surface. Let us recall 
the notion of \emph{resolution} of such a singularity:

\begin{definition}  \label{def:resol}  \index{resolution!of a surface singularity}
    Let $(S, s)$ be a normal surface singularity. A {\bf resolution} of it is a proper 
    bimeromorphic morphism $\pi : S^{\pi} \to S$ such that $S^{\pi}$ is smooth. 
    Its {\bf exceptional divisor} \index{divisor!exceptional} 
    $\boxed{E^{\pi}}$ is the reduced preimage $\pi^{-1}(s)$. 
    The resolution is {\bf good} if its exceptional divisor has normal crossings and all 
    its irreducible components are smooth. The {\bf dual graph} \index{dual graph!of a resolution} 
    $\boxed{\Gamma(\pi)}$ of the resolution 
    $\pi$ is the finite graph whose set of vertices $\boxed{\cP(\pi)}$ is the set of irreducible components 
    of $E^{\pi}$, two vertices being joined by an edge if and only if the corresponding 
    components intersect. 
\end{definition}

Every normal surface singularity admits resolutions and even good ones. 
This result, for which partial proofs 
appeared already at the end of the XIXth century, was proved first in the analytical context 
by Hirzebruch in his 1953 paper \cite{H 53}. His proof was inspired by previous works 
of Jung \cite{J 08} and Walker \cite{W 35}, done in an algebraic context.

Assume now that $\cF$ is a finite set of branches on $S$. It may be also seen as 
a reduced divisor on $S$, by thinking about their sum. The notion of \emph{embedded 
resolution} of $\cF$ is an analog of that of good resolution of $S$: 

\begin{definition}   \label{def:embres}    
    Let $(S,s)$ be a normal surface singularity and let $\pi : S^{\pi} \to S$ be a resolution of $S$. 
    If $A$ is a branch on $S$, then its {\bf strict transform by $\pi$} 
    \index{strict transform} is the closure inside 
    $S^{\pi}$ of the preimage $\pi^{-1}(A \setminus s)$.  
    Let $\cF$ be a finite set of branches on $S$. Its {\bf strict transform by $\pi$} 
    is the set or, depending on the context, the divisor 
    formed by the strict transforms of the branches of $\cF$. 
    The {\bf preimage $\boxed{\pi^{-1} \cF}$ 
    of $\cF$ by $\pi$} is the sum of its strict transform and of 
    the exceptional divisor of $\pi$.  The morphism $\pi$ is an 
    {\bf embedded resolution of $\cF$} \index{resolution!embedded} if it is a good resolution of $S$ and 
    the preimage of $\cF$ by $\pi$ is a normal crossings divisor. The {\bf dual graph} 
    $\boxed{\Gamma(\pi^{-1} \cF)}$
    of the preimage $\pi^{-1} \cF$ is defined similarly to the dual graph $\Gamma(\pi)$ of $\pi$, 
    taking into account all the irreducible components of $\pi^{-1} \cF$. 
\end{definition}

In the previous definition, the preimage $\pi^{-1} \cF$ of $\cF$ by $\pi$ is seen as a reduced 
divisor. We will see in Definition \ref{def:Mumftot} below that there is also a canonical 
way, due to Mumford, to define canonically a not necessarily reduced rational divisor supported 
by $\pi^{-1} \cF$, called the \emph{total transform} of $\cF$ by $\pi$, and denoted 
by $\pi^* \cF$.

The notion of dual graph of a resolution allows to define the following class of \emph{arborescent 
singularities}, whose name was introduced in the paper \cite{GBGPPP 18}, even if the class   
had appear before, for instance in Camacho's work \cite{C 88}: 

\begin{definition}  \label{def:arbsing}
     Let $S$ be a normal surface singularity. It is called {\bf arborescent} if the dual 
     graphs of its good resolutions are trees.  \index{singularity!arborescent}
\end{definition}

Remark that in the previous definition we ask nothing about the genera 
of the irreducible components of the exceptional divisors. 

By using the fact that any two resolutions are related by a sequence of blow ups and blow downs 
of their total spaces (see \cite[Thm. V.5.5]{H 77}), 
one sees that the dual graphs of all good resolutions are trees 
if and only if this is true for one of them. 

Consider now an \emph{arbitrary} branch $L$ on the normal surface singularity 
$S$. We may define the function $u_L$ by the same formula (\ref{eq:ultramrelbranch}) 
as in the case when both $S$ and $L$ were assumed smooth. Intersection numbers of 
branches still have a meaning, as was shown by Mumford. We will explain this  
in Section \ref{sec:mumfint} below (see Definition \ref{def:Mumfint}). 

The following generalization of Theorem \ref{thm:ulultram}  
both gives a characterization of arborescent singularities and 
extends Theorem \ref{thm:idtrees} to \emph{all} arborescent 
singularities $S$ and \emph{all} -- not necessarily smooth -- reference branches $L$ on them 
(recall that $\cB(S)$ denotes the set of branches on $S$):

\begin{theorem}  \label{thm:chararb}
      Let $S$ be a normal surface singularity 
      and $L \in \cB(S)$. Then:
         \begin{enumerate}
             \item \label{item:chararb} 
                   $u_L$ is ultrametric on $\cB(S) \setminus \{L\}$ 
                  if and only if $S$ is arborescent. 
             \item \label{item:chartreearb} 
                 In this case, for any finite set $\cF$ of branches 
                on $S$ not containing $L$, the rooted tree of the restriction of $u_L$ to 
                $\cF $ is isomorphic to the convex hull of $\cF \cup \{L \}$ in the dual graph 
                of the preimage of $\cF\cup \{L \}$ by any embedded resolution of 
                $\cF\cup \{L \}$, rooted at $L$. 
          \end{enumerate}
\end{theorem}

We do not prove in the present notes that if $u_L$ is an ultrametric on $\cB(S) \setminus \{L\}$, 
then $S$ is arborescent. The interested reader may find a proof of this fact in 
\cite[Sect. 1.6]{GBGPPPR 19}. The remaining implication of point (\ref{item:chararb}) and 
point (\ref{item:chartreearb}) of Theorem \ref{thm:chararb} are, taken together, a consequence 
of Theorem \ref{thm:ultrambv} below. For this reason, we do not give a separate proof of them, the 
rest of this paper being dedicated to the statement and a sketch of proof of 
Theorem \ref{thm:ultrambv}. The notion of \emph{brick-vertex tree} of a finite connected graph 
being crucial in this theorem, we dedicate next section to it. 

By combining Theorems \ref{thm:idtrees} and \ref{thm:chararb} one gets 
(see \cite[Thm. 112]{GBGPPP 18}): 

\begin{proposition} \label{prop:isotrees} 
   Whenever $S$ and $L$ are both smooth, 
  the Eggers-Wall tree $\Theta_L(D)$ of a plane curve singularity $D \hookrightarrow S$ 
  not containing $L$ is isomorphic to the convex hull of the strict transform of 
  $\cF(D) \cup \{L \}$ in the dual graph of its preimage by any of its embedded 
  resolutions. 
\end{proposition} 

A prototype of this fact was proved differently in the author's thesis \cite[Thm. 4.4.1]{PP 01}, 
then generalized in two different ways by Wall in \cite[Thm. 9.4.4]{W 04} 
(see also Wall's comments in \cite[Sect. 9.10]{W 04}) and by Favre and Jonsson 
in \cite[Prop. D.1]{FJ 04}.

\section{The brick-vertex tree of a connected graph}
\label{sec:brickvertex}

In this section we introduce the notion of \emph{brick-vertex tree} of a connected graph, 
which is crucial in order to state Theorem \ref{thm:ultrambv} below, the 
strongest known generalization of P\l oski's theorem.
\medskip

\begin{definition}   \label{def:graph}   \index{graph}
      A {\bf graph} is a compact cell complex of dimension $\leq 1$. If $\Gamma$ 
      is a graph, its set of vertices is denoted $\boxed{V(\Gamma)}$ and its set of edges 
      is denoted $\boxed{E(\Gamma)}$. 
\end{definition}

In the sequel it will be crucial to look at the vertices which disconnect a given graph:

\begin{definition}  \label{def:cut}   \index{cut-vertex}  \index{bridge} \index{separation in graphs}
      \index{graph!cut-vertex of} \index{graph!bridge of} 
      Let $\Gamma$ be a connected graph. A {\bf cut-vertex} of $\Gamma$ is a vertex 
      whose removal disconnects $\Gamma$. A {\bf bridge} of $\Gamma$ is an edge 
      such that the removal of its interior disconnects $\Gamma$. If $a,b,c$ are three 
      not necessarily distinct vertices of $\Gamma$, one says that $b$ 
      {\bf separates $a$ from $c$} if either $b \in \{a,c\}$ or if $a$ and $c$ belong to 
      different connected components of the topological space $\Gamma \: \setminus \: \{b\}$. 
\end{definition}

Note that an end of a bridge is a cut-vertex if and only if it has valency at least $2$ 
in $\Gamma$, that is, if and only if it is not  a  leaf of $\Gamma$. It will be important 
to distinguish the class of graphs which cannot be disconnected by the removal of one vertex, 
as well as the maximal graphs of this class contained in a given connected graph:

\begin{definition}   \label{def:insepblockbrick} \index{graph!inseparable}  
     \index{graph!block of}  \index{graph!brick of} 
      A connected graph is called {\bf inseparable} if it does not contain cut-vertices. 
      A {\bf block} of a connected graph $\Gamma$ is a maximal inseparable subgraph of it. 
      A {\bf brick} of $\Gamma$ is a block which is not a bridge. 
\end{definition}

Note that all the bridges of a connected graph are blocks of it. 

\begin{example}  \label{ex:cutbrickbloc}
    In Figure \ref{fig:Cutbridgeblocks} is represented a connected graph. 
    Its cut-vertices are surrounded in red. Its bridges are represented as black segments. 
    It has three bricks, the edges of each brick being colored in the same way.
\end{example}

\begin{figure}[h!]
\begin{center}
\begin{tikzpicture}[scale=0.5]

 \draw [-, color=black, thick](0,0) -- (-1,1) ;
 \draw [-, color=black, thick](0,0) -- (-1,-1) ;
  \node[draw,circle, inner sep=1.5pt,color=black, fill=black] at (-1,1){};
  \node[draw,circle, inner sep=1.5pt,color=black, fill=black] at (-1,-1){};
 
   \draw [-, color=black, thick](0,0) -- (2,0) ;
   
    \draw [-, color=blue, very thick](4,2) -- (2,0) ;
    \draw [-, color=blue, very thick](4,-2) -- (2,0) ;
    \draw [-, color=blue, very thick](4,-2) -- (4,2) ;
       \draw [-, color=blue, very thick](4,-2) -- (6,0) ;
    \draw [-, color=blue, very thick](6,0) -- (4,2) ;
    
     \draw [-, color=orange, very thick](2,-4) -- (4,-2) ;
    \draw [-, color=orange, very thick](4,-2) -- (6,-4) ;
    \draw [-, color=orange, very thick](2,-4) -- (6,-4) ;
      \node[draw,circle, inner sep=1.5pt,color=black, fill=black] at (4,2){};
  \node[draw,circle, inner sep=2.5pt,color=red, thick, fill=white] at (4,-2){};
  \node[draw,circle, inner sep=1.5pt,color=black, fill=black] at (4,-2){};
   \node[draw,circle, inner sep=2.5pt,color=red, thick, fill=white] at (0,0){};
  \node[draw,circle, inner sep=1.5pt,color=black, fill=black] at (0,0){};
  \node[draw,circle, inner sep=2.5pt,color=red, thick, fill=white] at (2,0){};
  \node[draw,circle, inner sep=1.5pt,color=black, fill=black] at (2,0){};
    \node[draw,circle, inner sep=1.5pt,color=black, fill=black] at (6,-4){};
    \node[draw,circle, inner sep=1.5pt,color=black, fill=black] at (2,-4){};
    \draw [-, color=black, thick](6,0) -- (9,0) ;
    \draw [-, color=black, thick](9,0) -- (9,1) ;
    \draw [-, color=black, thick](9,-1) -- (9,0) ;
    \draw [-, color=black, thick](9,0) -- (12,0) ;
    
    \node[draw,circle, inner sep=2.5pt,color=red, thick, fill=white] at (6,0){};
  \node[draw,circle, inner sep=1.5pt,color=black, fill=black] at (6,0){};
     \node[draw,circle, inner sep=2.5pt,color=red, thick, fill=white] at (9,0){};
  \node[draw,circle, inner sep=1.5pt,color=black, fill=black] at (9,0){};
    
      \draw [-, color=black, thick](8,-2) -- (9,-1) ;
    \draw [-, color=black, thick](9,-1) -- (10,-2) ;
     \node[draw,circle, inner sep=2.5pt,color=red, thick, fill=white] at (9,-1){};
  \node[draw,circle, inner sep=1.5pt,color=black, fill=black] at (9,-1){};
  \node[draw,circle, inner sep=1.5pt,color=black, fill=black] at (8,-2){};
  \node[draw,circle, inner sep=1.5pt,color=black,  fill=black] at (10,-2){};
    \node[draw,circle, inner sep=1.5pt,color=black, fill=black] at (9,1){};
    \draw [-, color=green, very thick](12,0) -- (14,-2) ;
    \draw [-, color=green, very thick](12,0) -- (14,2) ;
    \draw [-, color=green, very thick](14,-2) -- (14,2) ;
       \node[draw,circle, inner sep=2.5pt,color=red, thick, fill=white] at (12,0){};
       \node[draw,circle, inner sep=1.5pt,color=black, fill=black] at (12,0){};
   \node[draw,circle, inner sep=1.5pt,color=black, fill=black] at (14,2){};
  \node[draw,circle, inner sep=1.5pt,color=black, thick, fill=black] at (14,-2){};
 
 \end{tikzpicture}
\end{center}
\caption{A connected graph, its cut-vertices, its bridges and its bricks}
\label{fig:Cutbridgeblocks}
\end{figure}
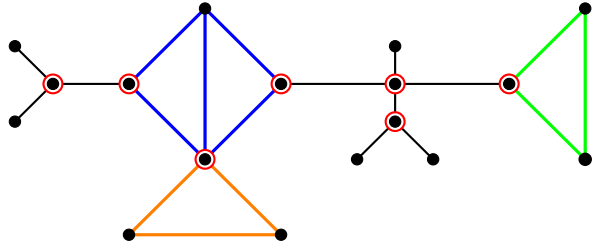

By replacing each brick of a connected graph by a star-shaped graph, one gets canonically a 
tree associated to the given graph: 

\begin{definition}    \label{def:bvtree}    \index{tree!brick-vertex}   
      The {\bf brick-vertex tree} $\cB\cV(\Gamma)$ of a connected graph $\Gamma$ 
      is the tree whose set of vertices is the union of the set of vertices of $\Gamma$ 
      and of a set of new {\bf brick-vertices} corresponding bijectively to the bricks of $\Gamma$,  
      its edges being either the bridges of $\Gamma$ or new edges 
      connecting each brick-vertex to the vertices of the corresponding brick. 
      Formally, this may be written as follows:
         \begin{itemize}
             \item   $V( \cB\cV(\Gamma) ) = V(\Gamma) \sqcup \{ \mathrm{bricks}  \:  
                          \mathrm{of} \: \Gamma\}$.
             \item   $ E( \cB\cV(\Gamma) ) = \{ \mathrm{bridges} \:  \mathrm{of} \: \Gamma\} 
                          \sqcup \{ [\overline{v}, \overline{b}], v \in V(\Gamma), b \mathrm{\:  is\:  a \: brick\:  of \:} 
                              \Gamma, v \in V(b)\}$. 
         \end{itemize}
      We denoted by $\boxed{\overline{v}}$ the vertex $v$ of $\Gamma$ when it is seen as a vertex of 
      $\cB\cV(\Gamma)$ and $\boxed{\overline{b}} \in V( \cB\cV(\Gamma) )$ the brick-vertex 
      representing the brick $b$ of $\Gamma$. 
\end{definition}

The notion of brick-vertex tree was introduced in \cite[Def. 1.34]{GBGPPPR 19}. It is strongly related 
to other notions introduced before either in general topology or in graph theory, as 
explained in \cite[Rems. 1.35, 2.50]{GBGPPPR 19}. 

Note that whenever $\Gamma$ is a tree, $ \cB\cV(\Gamma)$ is canonically isomorphic to it, 
as $\Gamma$ has no bricks.

\begin{example}  \label{ex:brickvertreeex}
    On the left side of Figure \ref{fig:Brickvertree} is repeated the graph $\Gamma$ of Figure 
    \ref{fig:Cutbridgeblocks}, with its cut-vertices and bricks emphasized. 
    On its right side is represented its associated 
    brick-vertex tree $\cB\cV(\Gamma)$. Each representative vertex of a brick is drawn with the 
    same color as its corresponding brick. The edges of $\cB\cV(\Gamma)$ which are not bridges of 
    $\Gamma$ are represented in magenta and thicker than the other edges.
\end{example}

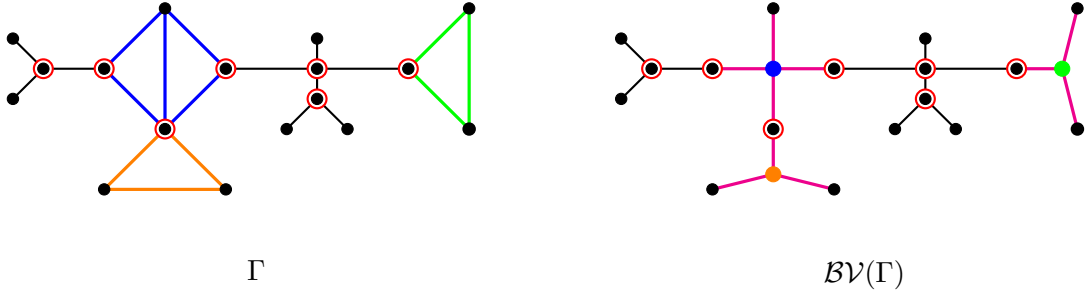
\begin{figure}[h!]
\begin{center}
\begin{tikzpicture}[scale=0.4]

   \draw [-, color=black, thick](0,0) -- (-1,1) ;
 \draw [-, color=black, thick](0,0) -- (-1,-1) ;
  \node[draw,circle, inner sep=1.5pt,color=black, fill=black] at (-1,1){};
  \node[draw,circle, inner sep=1.5pt,color=black, fill=black] at (-1,-1){};
 
   \draw [-, color=black, thick](0,0) -- (2,0) ;
   
    \draw [-, color=blue, very thick](4,2) -- (2,0) ;
    \draw [-, color=blue, very thick](4,-2) -- (2,0) ;
    \draw [-, color=blue, very thick](4,-2) -- (4,2) ;
       \draw [-, color=blue, very thick](4,-2) -- (6,0) ;
    \draw [-, color=blue, very thick](6,0) -- (4,2) ;
    
     \draw [-, color=orange, very thick](2,-4) -- (4,-2) ;
    \draw [-, color=orange, very thick](4,-2) -- (6,-4) ;
    \draw [-, color=orange, very thick](2,-4) -- (6,-4) ;
      \node[draw,circle, inner sep=1.5pt,color=black, fill=black] at (4,2){};
  \node[draw,circle, inner sep=2.5pt,color=red, thick, fill=white] at (4,-2){};
  \node[draw,circle, inner sep=1.5pt,color=black, fill=black] at (4,-2){};
   \node[draw,circle, inner sep=2.5pt,color=red, thick, fill=white] at (0,0){};
  \node[draw,circle, inner sep=1.5pt,color=black, fill=black] at (0,0){};
  \node[draw,circle, inner sep=2.5pt,color=red, thick, fill=white] at (2,0){};
  \node[draw,circle, inner sep=1.5pt,color=black, fill=black] at (2,0){};
    \node[draw,circle, inner sep=1.5pt,color=black, fill=black] at (6,-4){};
    \node[draw,circle, inner sep=1.5pt,color=black, fill=black] at (2,-4){};
    \draw [-, color=black, thick](6,0) -- (9,0) ;
    \draw [-, color=black, thick](9,0) -- (9,1) ;
    \draw [-, color=black, thick](9,-1) -- (9,0) ;
    \draw [-, color=black, thick](9,0) -- (12,0) ;
    
    \node[draw,circle, inner sep=2.5pt,color=red, thick, fill=white] at (6,0){};
  \node[draw,circle, inner sep=1.5pt,color=black, fill=black] at (6,0){};
     \node[draw,circle, inner sep=2.5pt,color=red, thick, fill=white] at (9,0){};
  \node[draw,circle, inner sep=1.5pt,color=black, fill=black] at (9,0){};
    
      \draw [-, color=black, thick](8,-2) -- (9,-1) ;
    \draw [-, color=black, thick](9,-1) -- (10,-2) ;
     \node[draw,circle, inner sep=2.5pt,color=red, thick, fill=white] at (9,-1){};
  \node[draw,circle, inner sep=1.5pt,color=black, fill=black] at (9,-1){};
  \node[draw,circle, inner sep=1.5pt,color=black, fill=black] at (8,-2){};
  \node[draw,circle, inner sep=1.5pt,color=black,  fill=black] at (10,-2){};
    \node[draw,circle, inner sep=1.5pt,color=black, fill=black] at (9,1){};
    \draw [-, color=green, very thick](12,0) -- (14,-2) ;
    \draw [-, color=green, very thick](12,0) -- (14,2) ;
    \draw [-, color=green, very thick](14,-2) -- (14,2) ;
       \node[draw,circle, inner sep=2.5pt,color=red, thick, fill=white] at (12,0){};
       \node[draw,circle, inner sep=1.5pt,color=black, fill=black] at (12,0){};
   \node[draw,circle, inner sep=1.5pt,color=black, fill=black] at (14,2){};
  \node[draw,circle, inner sep=1.5pt,color=black, thick, fill=black] at (14,-2){};

  \node [below] at (7,-6) {$\Gamma$};

  \begin{scope}[shift={(20,0)}]
  
   \draw [-, color=black, thick](0,0) -- (-1,1) ;
 \draw [-, color=black, thick](0,0) -- (-1,-1) ;
  \node[draw,circle, inner sep=1.5pt,color=black, fill=black] at (-1,1){};
  \node[draw,circle, inner sep=1.5pt,color=black, fill=black] at (-1,-1){};
 
   \draw [-, color=black, thick](0,0) -- (2,0) ;

    \draw [-, color=magenta, very thick](4,0) -- (2,0) ;
    \draw [-, color=magenta, very thick](4,0) -- (4,-2) ;
     \draw [-, color=magenta, very thick](4,0) -- (4,2) ;  
     \draw [-, color=magenta, very thick](4,0) -- (6,0) ;
     \node[draw,circle, inner sep=2pt,color=blue, fill=blue] at (4,0){};

     \draw [-, color=magenta, very thick](4,-3.5) -- (2,-4) ;
    \draw [-, color=magenta, very thick](4,-3.5) -- (4,-2) ;
     \draw [-, color=magenta, very thick](4,-3.5) -- (6,-4) ;
     \node[draw,circle, inner sep=2pt,color=orange, fill=orange] at (4,-3.5){};
    
      \node[draw,circle, inner sep=1.5pt,color=black, fill=black] at (4,2){};
  \node[draw,circle, inner sep=2.5pt,color=red, thick, fill=white] at (4,-2){};
  \node[draw,circle, inner sep=1.5pt,color=black, fill=black] at (4,-2){};
   \node[draw,circle, inner sep=2.5pt,color=red, thick, fill=white] at (0,0){};
  \node[draw,circle, inner sep=1.5pt,color=black, fill=black] at (0,0){};
  \node[draw,circle, inner sep=2.5pt,color=red, thick, fill=white] at (2,0){};
  \node[draw,circle, inner sep=1.5pt,color=black, fill=black] at (2,0){};
    \node[draw,circle, inner sep=1.5pt,color=black, fill=black] at (6,-4){};
    \node[draw,circle, inner sep=1.5pt,color=black, fill=black] at (2,-4){};
    \draw [-, color=black,  thick](6,0) -- (9,0) ;
    \draw [-, color=black,  thick](9,0) -- (9,1) ;
    \draw [-, color=black,  thick](9,-1) -- (9,0) ;
    \draw [-, color=black,  thick](9,0) -- (12,0) ;
    
    \node[draw,circle, inner sep=2.5pt,color=red, thick, fill=white] at (6,0){};
  \node[draw,circle, inner sep=1.5pt,color=black, fill=black] at (6,0){};
     \node[draw,circle, inner sep=2.5pt,color=red, thick, fill=white] at (9,0){};
  \node[draw,circle, inner sep=1.5pt,color=black, fill=black] at (9,0){};
    
    \draw [-, color=black, thick](8,-2) -- (9,-1) ;
    \draw [-, color=black, thick](9,-1) -- (10,-2) ;
     \node[draw,circle, inner sep=2.5pt,color=red, thick, fill=white] at (9,-1){};
  \node[draw,circle, inner sep=1.5pt,color=black, fill=black] at (9,-1){};
  \node[draw,circle, inner sep=1.5pt,color=black, fill=black] at (8,-2){};
  \node[draw,circle, inner sep=1.5pt,color=black,  fill=black] at (10,-2){};
    \node[draw,circle, inner sep=1.5pt,color=black, fill=black] at (9,1){};

     \draw [-, color=magenta, very thick](13.5,0) -- (12,0) ;
    \draw [-, color=magenta, very thick](13.5,0) -- (14,-2) ;
     \draw [-, color=magenta, very thick](13.5,0) -- (14,2) ;
    \node[draw,circle, inner sep=2pt,color=green, fill=green] at (13.5,0){};
    
       \node[draw,circle, inner sep=2.5pt,color=red, thick, fill=white] at (12,0){};
       \node[draw,circle, inner sep=1.5pt,color=black, fill=black] at (12,0){};
   \node[draw,circle, inner sep=1.5pt,color=black, fill=black] at (14,2){};
  \node[draw,circle, inner sep=1.5pt,color=black, thick, fill=black] at (14,-2){};

     \node [below] at (7,-6) {$\cB\cV(\Gamma)$};
     
   \end{scope}
 
 \end{tikzpicture}
\end{center}
\caption{The connected graph of Example \ref{ex:brickvertreeex} and its brick-vertex tree}
\label{fig:Brickvertree}
\end{figure}

The importance of the brick-vertex tree in our context stems from the 
following property of it (see \cite[Prop. 1.36]{GBGPPP 19}), formulated using the vocabulary introduced 
in Definition \ref{def:cut} and the notations introduced in Definition \ref{def:bvtree}:

\begin{proposition}  \label{prop:importbv}
   Let $\Gamma$ be a finite graph and $a,b,c \in V(\Gamma)$. Then $b$ separates 
   $a$ from $c$ in $\Gamma$ if and only if $\overline{b}$ separates 
   $\overline{a}$ from $\overline{c}$ in $\cB\cV(\Gamma)$. 
\end{proposition}

We are ready now to state the strongest known generalization of P\l oski's theorem 
(see Theorem \ref{thm:ultrambv} below).

\section{Our strongest generalization of P\l oski's theorem}  \label{sec:ultrambvtree}
\medskip

In this section we formulate Theorem \ref{thm:ultrambv}, which generalizes 
Theorem \ref{thm:idtrees} to \emph{all} normal surface singularities and \emph{all} 
branches on them, using the notion of \emph{brick-vertex tree} introduced in the previous section. 
\medskip

Recall that the notion of \emph{brick-vertex tree} of a connected graph was 
introduced in Definition \ref{def:bvtree}. A fundamental property of normal 
surface singularities is that the dual graphs of their resolutions are connected 
(which is a particular case of the so-called \emph{Zariski's main theorem}, 
whose statement may be found in \cite[Thm. V.5.2]{H 77}). 
This implies that the dual graph of the preimage (see Definition \ref{def:embres}) 
of any finite set of branches on such a singularity is also connected. 
Therefore, one may speak about its corresponding brick-vertex tree. 
The {\bf convex hull} of a finite set of vertices of it is the union of the segments 
which join them pairwise.  \index{convex hull!in a tree}

Here is the announced generalization of Theorem \ref{thm:idtrees}, which is a slight 
reformulation of \cite[Thm. 1.42]{GBGPPPR 19}: 

\begin{theorem}      \label{thm:ultrambv}
     Let $S$ be a normal surface singularity. Consider a finite set $\cF$ of branches 
     on it and an embedded resolution $\pi: S^{\pi} \to S$ of $\cF$. Let $\Gamma$ be the 
     dual graph of the preimage $\pi^{-1} \cF$ of $\cF$ by $\pi$. Assume that the convex hull 
     $\mathrm{Conv}_{\cB\cV(\Gamma)}(\cF)$ of the strict transform of $\cF$ by $\pi$ 
     in the brick-vertex tree $\cB\cV(\Gamma)$ does not contain brick-vertices 
     of valency at least $4$ in $\mathrm{Conv}_{\cB\cV(\Gamma)}(\cF)$. Then 
     for all $L \in \cF$, the restriction of $u_L$ to $\cF \: \setminus \:  \{L\}$ 
     is an ultrametric and the corresponding rooted tree is isomorphic to 
     $\mathrm{Conv}_{\cB\cV(\Gamma)}(\cF)$, rooted at $L$. 
\end{theorem}

\begin{example}  \label{ex:brickvertree}
     Assume that the dual graph $\Gamma$ of $\pi^{-1} \cF$ is as shown on the left side 
     of Figure \ref{fig:Hypsatisf}. The vertices representing the strict transforms of 
     the branches of the set $\cF$ are drawn arrowheaded. 
     Note that the subgraph which is the dual graph of 
     the exceptional divisor is the same as the graph of Figure \ref{fig:Cutbridgeblocks}. 
     On the right side of Figure \ref{fig:Hypsatisf} is represented using thick red 
     segments the convex 
     hull $\mbox{Conv}_{\cB\cV(\Gamma)}(\cF)$. We see that the hypothesis of 
     Theorem \ref{thm:ultrambv} is satisfied. Indeed, the convex hull contains only 
     two brick-vertices, which are of valency $2$ and $3$ in $\mbox{Conv}_{\cB\cV(\Gamma)}(\cF)$. 
     Note that the blue one is of valency $4$ in the dual graph $\Gamma$, which shows 
     the importance of looking at the valency in the convex hull 
     $\mbox{Conv}_{\cB\cV(\Gamma)}(\cF)$, not in $\Gamma$. 
     
\end{example}

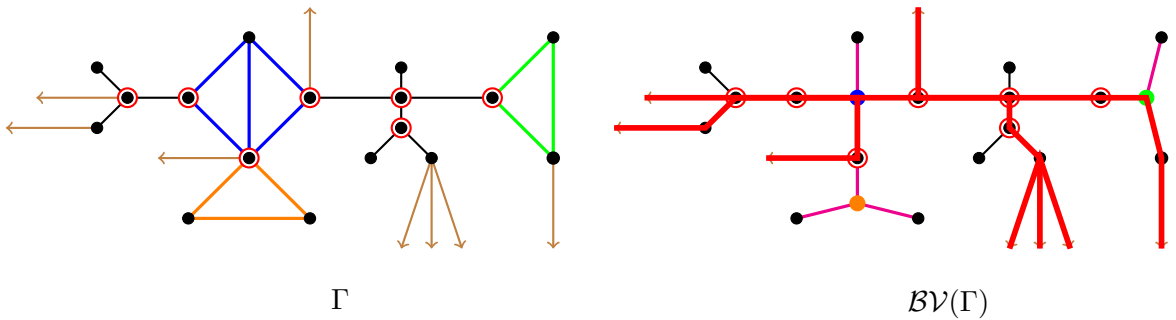
\begin{figure}[h!]
\begin{center}
\begin{tikzpicture}[scale=0.4]

     \draw [->, color=brown, thick] (0,0)--(-3,0);
     \draw [->, color=brown, thick] (-1,-1)--(-4,-1);
     \draw [->, color=brown, thick] (4,-2)--(1,-2);
     \draw [->, color=brown, thick] (6,0)--(6,3);
        \draw [->, color=brown, thick] (10,-2)--(9,-5);
        \draw [->, color=brown, thick] (10,-2)--(10,-5);
        \draw [->, color=brown, thick] (10,-2)--(11,-5);
    \draw [->, color=brown, thick] (14,-2)--(14,-5);
  
   \draw [-, color=black, thick](0,0) -- (-1,1) ;
 \draw [-, color=black, thick](0,0) -- (-1,-1) ;
  \node[draw,circle, inner sep=1.5pt,color=black, fill=black] at (-1,1){};
  \node[draw,circle, inner sep=1.5pt,color=black, fill=black] at (-1,-1){};
 
   \draw [-, color=black, thick](0,0) -- (2,0) ;
   
    \draw [-, color=blue, very thick](4,2) -- (2,0) ;
    \draw [-, color=blue, very thick](4,-2) -- (2,0) ;
    \draw [-, color=blue, very thick](4,-2) -- (4,2) ;
       \draw [-, color=blue, very thick](4,-2) -- (6,0) ;
    \draw [-, color=blue, very thick](6,0) -- (4,2) ;
    
     \draw [-, color=orange, very thick](2,-4) -- (4,-2) ;
    \draw [-, color=orange, very thick](4,-2) -- (6,-4) ;
    \draw [-, color=orange, very thick](2,-4) -- (6,-4) ;
      \node[draw,circle, inner sep=1.5pt,color=black, fill=black] at (4,2){};
  \node[draw,circle, inner sep=2.5pt,color=red, thick, fill=white] at (4,-2){};
  \node[draw,circle, inner sep=1.5pt,color=black, fill=black] at (4,-2){};
   \node[draw,circle, inner sep=2.5pt,color=red, thick, fill=white] at (0,0){};
  \node[draw,circle, inner sep=1.5pt,color=black, fill=black] at (0,0){};
  \node[draw,circle, inner sep=2.5pt,color=red, thick, fill=white] at (2,0){};
  \node[draw,circle, inner sep=1.5pt,color=black, fill=black] at (2,0){};
    \node[draw,circle, inner sep=1.5pt,color=black, fill=black] at (6,-4){};
    \node[draw,circle, inner sep=1.5pt,color=black, fill=black] at (2,-4){};
    \draw [-, color=black, thick](6,0) -- (9,0) ;
    \draw [-, color=black, thick](9,0) -- (9,1) ;
    \draw [-, color=black, thick](9,-1) -- (9,0) ;
    \draw [-, color=black, thick](9,0) -- (12,0) ;
    
    \node[draw,circle, inner sep=2.5pt,color=red, thick, fill=white] at (6,0){};
  \node[draw,circle, inner sep=1.5pt,color=black, fill=black] at (6,0){};
     \node[draw,circle, inner sep=2.5pt,color=red, thick, fill=white] at (9,0){};
  \node[draw,circle, inner sep=1.5pt,color=black, fill=black] at (9,0){};
    
      \draw [-, color=black, thick](8,-2) -- (9,-1) ;
    \draw [-, color=black, thick](9,-1) -- (10,-2) ;
     \node[draw,circle, inner sep=2.5pt,color=red, thick, fill=white] at (9,-1){};
  \node[draw,circle, inner sep=1.5pt,color=black, fill=black] at (9,-1){};
  \node[draw,circle, inner sep=1.5pt,color=black, fill=black] at (8,-2){};
  \node[draw,circle, inner sep=1.5pt,color=black,  fill=black] at (10,-2){};
    \node[draw,circle, inner sep=1.5pt,color=black, fill=black] at (9,1){};
    \draw [-, color=green, very thick](12,0) -- (14,-2) ;
    \draw [-, color=green, very thick](12,0) -- (14,2) ;
    \draw [-, color=green, very thick](14,-2) -- (14,2) ;
       \node[draw,circle, inner sep=2.5pt,color=red, thick, fill=white] at (12,0){};
       \node[draw,circle, inner sep=1.5pt,color=black, fill=black] at (12,0){};
   \node[draw,circle, inner sep=1.5pt,color=black, fill=black] at (14,2){};
  \node[draw,circle, inner sep=1.5pt,color=black, thick, fill=black] at (14,-2){};

  \node [below] at (7,-6) {$\Gamma$};

  \begin{scope}[shift={(20,0)}]

  \draw [->, color=brown, thick] (0,0)--(-3,0);
     \draw [->, color=brown, thick] (-1,-1)--(-4,-1);
     \draw [->, color=brown, thick] (4,-2)--(1,-2);
     \draw [->, color=brown, thick] (6,0)--(6,3);
        \draw [->, color=brown, thick] (10,-2)--(9,-5);
        \draw [->, color=brown, thick] (10,-2)--(10,-5);
        \draw [->, color=brown, thick] (10,-2)--(11,-5);
    \draw [->, color=brown, thick] (14,-2)--(14,-5);
  
   \draw [-, color=black, thick](0,0) -- (-1,1) ;
 \draw [-, color=black, thick](0,0) -- (-1,-1) ;
  \node[draw,circle, inner sep=1.5pt,color=black, fill=black] at (-1,1){};
  \node[draw,circle, inner sep=1.5pt,color=black, fill=black] at (-1,-1){};
 
   \draw [-, color=black, thick](0,0) -- (2,0) ;

    \draw [-, color=magenta, very thick](4,0) -- (2,0) ;
    \draw [-, color=magenta, very thick](4,0) -- (4,-2) ;
     \draw [-, color=magenta, very thick](4,0) -- (4,2) ;  
     \draw [-, color=magenta, very thick](4,0) -- (6,0) ;
     \node[draw,circle, inner sep=2pt,color=blue, fill=blue] at (4,0){};

     \draw [-, color=magenta, very thick](4,-3.5) -- (2,-4) ;
    \draw [-, color=magenta, very thick](4,-3.5) -- (4,-2) ;
     \draw [-, color=magenta, very thick](4,-3.5) -- (6,-4) ;
     \node[draw,circle, inner sep=2pt,color=orange, fill=orange] at (4,-3.5){};
    
      \node[draw,circle, inner sep=1.5pt,color=black, fill=black] at (4,2){};
  \node[draw,circle, inner sep=2.5pt,color=red, thick, fill=white] at (4,-2){};
  \node[draw,circle, inner sep=1.5pt,color=black, fill=black] at (4,-2){};
   \node[draw,circle, inner sep=2.5pt,color=red, thick, fill=white] at (0,0){};
  \node[draw,circle, inner sep=1.5pt,color=black, fill=black] at (0,0){};
  \node[draw,circle, inner sep=2.5pt,color=red, thick, fill=white] at (2,0){};
  \node[draw,circle, inner sep=1.5pt,color=black, fill=black] at (2,0){};
    \node[draw,circle, inner sep=1.5pt,color=black, fill=black] at (6,-4){};
    \node[draw,circle, inner sep=1.5pt,color=black, fill=black] at (2,-4){};
    \draw [-, color=black, thick](6,0) -- (9,0) ;
    \draw [-, color=black, thick](9,0) -- (9,1) ;
    \draw [-, color=black,  thick](9,-1) -- (9,0) ;
    \draw [-, color=black,  thick](9,0) -- (12,0) ;
    
    \node[draw,circle, inner sep=2.5pt,color=red, thick, fill=white] at (6,0){};
  \node[draw,circle, inner sep=1.5pt,color=black, fill=black] at (6,0){};
     \node[draw,circle, inner sep=2.5pt,color=red, thick, fill=white] at (9,0){};
  \node[draw,circle, inner sep=1.5pt,color=black, fill=black] at (9,0){};
    
    \draw [-, color=black, thick](8,-2) -- (9,-1) ;
    \draw [-, color=black, thick](9,-1) -- (10,-2) ;
     \node[draw,circle, inner sep=2.5pt,color=red, thick, fill=white] at (9,-1){};
  \node[draw,circle, inner sep=1.5pt,color=black, fill=black] at (9,-1){};
  \node[draw,circle, inner sep=1.5pt,color=black, fill=black] at (8,-2){};
  \node[draw,circle, inner sep=1.5pt,color=black,  fill=black] at (10,-2){};
    \node[draw,circle, inner sep=1.5pt,color=black, fill=black] at (9,1){};

     \draw [-, color=magenta, very thick](13.5,0) -- (12,0) ;
    \draw [-, color=magenta, very thick](13.5,0) -- (14,-2) ;
     \draw [-, color=magenta, very thick](13.5,0) -- (14,2) ;
    \node[draw,circle, inner sep=2pt,color=green, fill=green] at (13.5,0){};
    
       \node[draw,circle, inner sep=2.5pt,color=red, thick, fill=white] at (12,0){};
       \node[draw,circle, inner sep=1.5pt,color=black, fill=black] at (12,0){};
   \node[draw,circle, inner sep=1.5pt,color=black, fill=black] at (14,2){};
  \node[draw,circle, inner sep=1.5pt,color=black, thick, fill=black] at (14,-2){};
  
   \draw [-, line width=0.75mm, color=red] (-3,0)--(13.5,0)--(14,-2)--(14,-5);
   \draw [-, line width=0.75mm,  color=red] (-4,-1)--(-1,-1)--(0,0)--(4,0)--(4,-2)--(1,-2);
   \draw [-, line width=0.75mm, color=red] (6,3)--(6,0)--(9,0)--(9,-1)--(10,-2)--(10,-5);
   \draw [-, line width=0.75mm, color=red] (9,-5)--(10,-2)--(11,-5);

     \node [below] at (7,-6) {$\cB\cV(\Gamma)$};
     
   \end{scope}
 
 \end{tikzpicture}
\end{center}
\caption{An example where the hypothesis of Theorem \ref{thm:ultrambv} is satisfied}
\label{fig:Hypsatisf}
\end{figure}

As shown in \cite[Ex. 1.44]{GBGPPPR 19}, the condition about valency is not necessary 
in general for $u_L$ to be an ultrametric on $\cF  \: \setminus \:  \{L\}$.

Note that we have expressed Theorem \ref{thm:ultrambv} in a slightly different form 
than the equivalent Theorem D of the introduction. Namely, we included $L$ 
in the branches of $\cF$. This formulation emphasizes the symmetry of the situation: 
all the choices of reference branch inside $\cF$ lead to the same tree, only the root 
being changed. In fact, we will obtain Theorem \ref{thm:ultrambv} 
as a consequence of Theorem \ref{thm:reformangdist}, in which no branch plays 
any more a special role.

Before that, we will explain in the next section Mumford's 
definition of intersection number of two curve singularities drawn on an arbitrary 
normal surface singularity, which allows to define in turn the functions $u_L$ appearing 
in the statement of Theorem \ref{thm:ultrambv}.

\section{Mumford's intersection theory}
\label{sec:mumfint}

In this section we explain Mumford's definition of intersection number of 
Weil divisors on a normal surface singularity, introduced in his 1961 paper \cite{M 61}. 
It is based on Theorem \ref{thm:negdef}, stating that the intersection form of any resolution of a 
normal surface singularity is negative definite. This theorem being 
fundamental for the study of surface singularities, we present a detailed proof of it. 
\medskip

Let $\pi : S^{\pi} \to S$ be a resolution of the normal surface 
singularity $S$. Denote by $(E_u)_{u \in \cP(\pi)}$ the collection of 
irreducible components of the exceptional divisor $E^{\pi}$ of $\pi$ 
(see Definition \ref{def:resol}). 

Denote by:
    \[ \boxed{\cE(\pi)_{\R}} := \bigoplus_{u \in \cP(\pi)} \R E_u \]
 the real vector space freely generated by those prime divisors, that is, 
 the space of real divisors supported by $E^{\pi}$. 
 It is endowed with a symmetric bilinear form $(D_1, D_2) \to D_1 \cdot D_2$ 
 given by intersecting the corresponding compact cycles on $S^{\pi}$. 
 We call it {\bf the intersection form}. Its following fundamental property was proved by 
 Du Val \cite{V 44} and Mumford \cite{M 61}:

 \begin{theorem}   \label{thm:negdef}
    The intersection form on $ \cE(\pi)_{\R}$ is negative definite. 
 \end{theorem} 
 
 \begin{proof}
     The following proof is an expansion of that given by Mumford in \cite{M 61}. 
     
    The singularity $S$ being normal, the exceptional divisor $E^{\pi}$ is connected 
    (this is a particular case of \emph{Zariski's main theorem}, see \cite[Thm. V.5.2]{H 77}). 
    Therefore: 
       \begin{equation}  \label{eq:conngraph}
          \mbox{The dual graph } \Gamma(\pi) \mbox{ is connected.} 
       \end{equation}
    Consider any germ of holomorphic function $f$ on $(S,s)$, vanishing at $s$, and 
    look at the divisor of its lift to the surface $S^{\pi}$:
       \begin{equation}   \label{eq:liftfunct}
            (\pi^* f) = \sum_{u \in \cP(\pi)} a_u E_u + (\pi^* f)_{str}.
       \end{equation}
     Here $ (\pi^* f)_{str}$ denotes the strict transform of the divisor defined by $f$ on $S$. 
     Denote also:
        \begin{equation}   \label{eq:newnot}
           \left\{ \begin{array}{l}
                        \boxed{e_u} := a_u E_u  \in  \cE(\pi)_{\R}, \mbox{ for all } u \in \cP(\pi), \\
                        \boxed{\sigma} : = \displaystyle{ \sum_{u \in \cP(\pi)} e_u \in  \cE(\pi)_{\R}}. 
                    \end{array} \right.
       \end{equation}
     As $f$ vanishes at the point $s$, its lift $\pi^* f$ vanishes along each component $E_u$ 
     of $E^{\pi}$, therefore $a_u >0$ for every $u \in \cP(u)$. We deduce that $(e_u)_{u \in \cP(\pi)}$ 
     is a basis of $\cE(\pi)_{\R}$ and that:
       \begin{equation}   \label{eq:ineqint}
            e_u \cdot e_v \geq 0, \mbox{ for all } u,v \in \cP(\pi) \mbox{ such that } u \neq v.
       \end{equation}
     The divisor $(\pi^* f)$ being principal, its associated line bundle is trivial. Therefore:  
       \begin{equation} \label{eq:intvanish} 
            (\pi^* f) \cdot E_u =0 \mbox{ for every } u \in  \cP(\pi), 
       \end{equation}
     because this intersection number is equal by definition to the degree of the restriction 
     of this line bundle to the curve $E_u$. By combining the relations (\ref{eq:liftfunct}), 
     (\ref{eq:newnot}) and (\ref{eq:intvanish}), we deduce that:
         \begin{equation}   \label{eq:equalint}
             \sigma \cdot e_u = - a_u  (\pi^* f)_{str} \cdot E_u, \mbox{ for every } u \in  \cP(\pi). 
         \end{equation}
     As the germ of effective divisor $(\pi^* f)_{str}$ along $E^{\pi}$ has no components of 
     $E^{\pi}$ in its support, the intersection numbers $(\pi^* f)_{str} \cdot E_u$ are all non-negative. 
     Moreover, at least one of them is positive, because the divisor $(\pi^* f)_{str}$ is non-zero. 
     By combining this fact with relations (\ref{eq:equalint}) and with the inequalities $a_u > 0$, 
     we get:
        \begin{equation}   \label{eq:ineqintsigma}
           \left\{ 
             \begin{array}{l}
                 \sigma \cdot e_u \leq 0, \mbox{ for every } u \in  \cP(\pi), \\
                 \mbox{there exists } u_0 \in \cP(\pi) \mbox{ such that } \sigma \cdot e_{u_0} < 0. 
             \end{array} 
                \right.  
         \end{equation}
         
         Consider now an arbitrary element $\tau \in \cE(\pi)_{\R} \:  \setminus \:  \{0\}$. 
         One may develop it in the basis $(e_u)_{u \in \cP(\pi)}$:
             \begin{equation} \label{eq:exptau} 
                   \tau = \sum_{u \in \cP(\pi)} x_u e_u . 
             \end{equation}
             
        We will show that $\tau^2 < 0$. As $\tau$ was chosen as an arbitrary non-zero 
   vector, this will imply that the intersection form on $ \cE(\pi)_{\R}$ is indeed negative definite. 
   The trick is to express the self-intersection $\tau^2$ using the expansion (\ref{eq:exptau}), 
   then to develop it 
   by bilinearity and to replace the vectors $e_u$ by $\sigma - \sum_{v \neq u} e_v$ in a precise 
   place:
     \[ \begin{array}{ll}
            \tau^2 &  = \displaystyle{ \left(\sum_{u} x_u e_u \right)^2 =} \\ 
              &   = \displaystyle{ \sum_u x_u^2 e_u^2 + 2 \sum_{u < v} x_u x_v e_u \cdot e_v =} \\
              &  =  \displaystyle{\sum_u x_u^2 \left(\sigma - \sum_{v \neq u} e_v \right) \cdot e_u + 
                    2 \sum_{u < v} x_u x_v e_u \cdot e_v =} \\
              &  = \displaystyle{\sum_u  x_u^2 (\sigma \cdot e_u)   - \sum_{u \neq v} x_u^2 e_u \cdot e_v + 
                       2 \sum_{u < v} x_u x_v e_u \cdot e_v =} \\ 
             &  = \displaystyle{\sum_u  x_u^2 (\sigma \cdot e_u)   - \sum_{u < v} (x_u - x_v)^2 e_u \cdot e_v.}
         \end{array} \]
      We got the equality:
      
    \begin{equation}  \label{eq:rewriting}
        \tau^2 = \displaystyle{\sum_u  x_u^2 (\sigma \cdot e_u)   - \sum_{u < v} (x_u - x_v)^2 e_u \cdot e_v.}
    \end{equation}
     Using the inequalities (\ref{eq:ineqint}) and (\ref{eq:ineqintsigma}), we deduce that  its right-hand 
     side is non-positive, therefore the intersection form is negative semi-definite. 
     
     It remains to show that $\tau^2 < 0$. Assume by contradiction that $\tau^2=0$. 
     Equality (\ref{eq:rewriting}) shows that the following equalities are simultaneously satisfied: 
          \begin{equation} \label{eq:left}
              \sum_u  x_u^2 (\sigma \cdot e_u)  = 0, 
          \end{equation}
          
          \begin{equation} \label{eq:right}
              (x_u - x_v)^2 e_u \cdot e_v =0, \mbox{ for all } u < v. 
          \end{equation}
      The relations (\ref{eq:right}) imply that $x_u = x_v$ whenever $e_u \cdot e_v >0$. 
      As $e_u = a_u E_u$ with $a_u >0$, the inequality $e_u \cdot e_v >0$ is equivalent 
      with $E_u \cdot E_v >0$, that is, with the fact that $[u,v]$ is an edge of the dual graph 
      $\Gamma(\pi)$. This dual graph being connected (see (\ref{eq:conngraph})), we see 
      that $x_u = x_v$ for all $u,v \in \cP(\pi)$. Consider now an index $u_0$ satisfying 
      the second condition of relations (\ref{eq:ineqintsigma}). Equation (\ref{eq:left}) 
      implies that $x_{u_0}=0$. Therefore all the coefficients $x_u$ vanish, which contradicts 
      the hypothesis that $\tau \neq 0$. 
 \end{proof}

As a consequence of Theorem \ref{thm:negdef}, 
one may define the {\bf dual basis} $(\boxed{E^{\vee}_u})_{u \in \cP(\pi)}$ 
of the basis $(E_u)_{u \in \cP(\pi)}$ by the following relations, in which 
$ \delta_{uv}$ denotes Kronecker's delta-symbol:
   \begin{equation}  \label{eq:expbasis}
        E^{\vee}_u \cdot E_v = \delta_{uv}, \:   \mbox{for all} \:  (u,v) \in \cP(\pi)^2. 
    \end{equation}
   
   By associating to each prime divisor $E_u$ 
the corresponding valuation of the local ring $\cO_{S,s}$, computing the orders  
of vanishing along $E_u$ of the pull-backs $\pi^* f$ of the functions 
$f \in \cO_{S,s}$, one injects the set $\cP(\pi)$ in the set of 
real-valued valuations of $\cO_{S,s}$. This allows to see the index $u$ of $E_u$ 
as a valuation. Such valuations are called {\em divisorial}. If $u$ denotes a divisorial 
valuation, it has a center on any resolution, which is either a point or an irreducible 
component of the exceptional divisor. In the second case, one says that 
the valuation {\bf appears in the resolution}. 
   The following notion, inspired by approaches of Favre $\&$ Jonsson 
   \cite[App. A]{FJ 04} and \cite[Sect. 7.3.6]{J 15}, was introduced in 
   \cite[Def. 1.6]{GBGPPPR 19}:

\begin{definition}   \label{def:bracket}   \index{bracket of two divisorial valuations}
    Let $u,v$ be two divisorial valuations of $S$. Consider a resolution 
    of $S$ in which both $u$ and $v$ appear. Then their {\bf bracket} is defined by:
       \[\boxed{\langle u, v \rangle} := - E^{\vee}_u \cdot E^{\vee}_v. \]
\end{definition}

The bracket $\boxed{\langle u, v \rangle}$ may be interpreted as the intersection 
number of two Weil divisors on $S$ associated to the divisors $E_u$ and $E_v$ 
(see Proposition \ref{prop:intbra} below). As a consequence, it is well-defined. 
That is, if the divisorial valuations $u,v$ are fixed, then their bracket does not depend 
on the resolution in which they appear. This fact  
may be also proved using the property that any two resolutions of $S$ are related by a sequence 
of blow ups and blow downs (see \cite[Prop. 1.5]{GBGPPPR 19}). 

It is a consequence of Theorem \ref{thm:negdef} that the brackets are all non-negative 
(see \cite[Prop. 1.4]{GBGPPPR 19}). Moreover, by the Cauchy-Schwarz inequality applied 
to the opposite of the intersection form:

\begin{lemma}  \label{lem:CS}
      For every $a,b \in \cP(\pi)$:
         \[ \langle a,b\rangle^2 \leq \langle a, a \rangle \langle b,b\rangle, \]
       with equality if and only if $a=b$. 
\end{lemma}

Let now $D$ be a Weil divisor on $S$, that is, a formal sum of branches on $S$. 
If $D$ is principal, that is, the divisor $(f)$ of a meromorphic germ on $S$, then 
one may lift it to a resolution $S^{\pi}$ as the principal divisor $(\pi^* f)$. This divisor 
decomposes as the sum of an exceptional part $(\pi^* D)_{ex}$ supported by $E^{\pi}$ and 
the strict transform of $D$. The crucial property of the lift $(\pi^* f)$, already used in the 
proof of Theorem \ref{thm:negdef} (see relation (\ref{eq:intvanish})), is that its intersection 
numbers with all the components $E_u$ of $E^{\pi}$ vanish. In \cite[Sect. II (b)]{M 61}, 
Mumford imposed this property in order to define a lift $\pi^* D$ for \emph{any} 
Weil divisor $D$ on $S$:

\begin{definition}  \label{def:Mumftot}   \index{total transform of a divisor} 
    Let $D$ be a Weil divisor on $S$. Its {\bf total transform} $\boxed{\pi^* D}$ is the 
    unique sum $\boxed{(\pi^* D)_{ex}}  + \boxed{(\pi^*D)_{str}}$ such that:
        \begin{enumerate}
            \item  $(\pi^* D)_{ex} \in  \cE(\pi)_{\Q}$.  
            \item   $(\pi^* D)_{str}$ is the strict transform of $D$ by $\pi$. 
            \item   \label{point:zeroint}   
                  $(\pi^* D) \cdot E_u =0$ for all $u \in \cP(\pi)$. 
        \end{enumerate}
     The divisor $(\pi^* D)_{ex}$ supported by the exceptional divisor of $\pi$ 
     is the {\bf exceptional transform} of $D$ by $\pi$.   \index{exceptional transform of a divisor} 
\end{definition}

The divisor $\pi^*D$ is well-defined, as results from Theorem \ref{thm:negdef}. 
The point is to show that $(\pi^* D)_{ex}$ exists and is unique with the 
property (\ref{point:zeroint}). Write it as a sum $\sum_{v \in \cP(\pi)} x_v E_v$. The last condition of 
Definition \ref{def:Mumftot} may be written as the system: 
   \[  \sum_{v \in \cP(\pi)}  (E_v \cdot E_u) x_v = - (\pi^*D)_{str} \cdot E_u, \:   
          \mbox{ for all } \:   u \in \cP(\pi).\]
This is a square linear system in the unknowns $x_v$, whose matrix is the matrix of the 
intersection form in the basis $(E_u)_{u \in \cP(\pi)}$. As the intersection form is negative 
definite, it is non-degenerate, therefore this system has a unique solution. Moreover, 
all its coefficients being integers, its solution has rational coordinates, which shows that 
$(\pi^* D)_{ex} \in  \cE(\pi)_{\Q}$ . 

Using Definition \ref{def:Mumftot} and the standard definition of intersection numbers on smooth 
surfaces recalled in Section \ref{sec:intmult},  Mumford defined in the following way 
in \cite[Sect. II (b)]{M 61} 
the intersection number of two Weil divisors on $S$: 

\begin{definition}    \label{def:Mumfint}   \index{intersection number!Mumford's} 
      \index{Mumford's!intersection number}
    Let $A,B$ be two Weil divisors on $S$ without common components, 
    and $\pi$ be a resolution of $S$. 
    Then the {\bf intersection number} of $A$ and $B$ is defined by: 
           \[   \boxed{A \cdot B} := \pi^* A \cdot \pi^* B.  \]
\end{definition}

Using the fact that any two resolutions of $S$ are related by a sequence of blow ups and 
blow downs (see \cite[Thm. V.5.5]{H 77}), 
it may be shown that the previous notion is independent of the choice of 
resolution, similarly to that of bracket of two divisorial valuations introduced in 
Definition \ref{def:bracket}. In particular, if $S$ is smooth, one may choose $\pi$ 
to be the identity. This shows that in this case Mumford's definition gives the same 
intersection number as the standard Definition \ref{def:intnumb}.

\begin{example} \label{ex:intnumb}
    Let $S$ be the germ at the origin $0$ of the quadratic cone $Z(x^2 + y^2 + z^2) \hookrightarrow \C^3$ 
    (it is the so-called $\mathbf{A_1}$ surface singularity). Let $A$ and $B$ be the germs at $0$ of two 
    distinct generating lines of the cone. One may resolve $S$ by blowing up $0$. This 
    morphism $\pi : S^{\pi} \to S$ separates all the generators, therefore it is an embedded resolution 
    of $\{A, B\}$. The exceptional divisor of $\pi$ is 
    the projectivisation of the cone, that is, it is a smooth rational curve $E$. Its self-intersection 
    number is the opposite of the degree of the curve seen embedded in the projectivisation 
    of the ambient space $\C^3$. Therefore, $E^2 = -2$. Let us compute the total transform 
    $\pi^* A = (\pi^* A)_{str} + x E$. The imposed constraint $\pi^* A  \cdot E =0$ becomes 
    $1 - 2x =0$, therefore $x = 1/2$. We have used the fact that the strict transform 
    $(\pi^* A)_{str}$ of $A$ by $\pi$ is smooth and transversal to $E$, which implies that 
    $(\pi^* A)_{str} \cdot E = 1$. 
    
    We obtained $\pi^* A = (\pi^* A)_{str} + (1/2) E$ and similarly, 
     $\pi^* B = (\pi^* B)_{str} + (1/2) E$. Using Definition \ref{def:Mumfint}, we get:
       \[  \begin{array}{ll}
                A \cdot B &  = \pi^* A \cdot \pi^* B = \\ 
                  &    =  \left(  (\pi^* A)_{str} + (1/2) E \right) \cdot   \left(  (\pi^* B)_{str} + (1/2) E \right) = \\
                  &    = (\pi^* A)_{str} \cdot (\pi^* B)_{str}  + (1/2) ( (\pi^* A)_{str} + (\pi^* B)_{str} ) \cdot E + 
                                      (1/2)^2 E^2 = \\
                   &   = 0 + (1/2) \cdot 2 +  (1/2)^2 \cdot (-2) =  \\
                   &    = 1/2. 
            \end{array}  \]
            
     This example shows in particular that \emph{the intersection number of two curve singularities 
     depends on the normal surface singularity on which it is computed}. Indeed, 
     the branches $A$ and  $B$ are also contained in a smooth surface (any two generators 
     of the quadratic cone are obtained as the intersection of the cone with a plane passing 
     through its vertex). In such a surface, their intersection number is $1$ instead of $1/2$. 
\end{example}

Definition \ref{def:Mumfint} allows to give the following interpretation of the notion of bracket 
introduced in Definition \ref{def:bracket} (see \cite[Prop. 1.11]{GBGPPPR 19}):

\begin{proposition}    \label{prop:intbra}
  Let $A, B$ be two distinct branches on $S$. Consider an embedded resolution $\pi$ 
  of their sum. Denote by $E_a, E_b$ the components of the exceptional divisor 
  $E^{\pi}$ which are intersected by the strict transforms $(\pi^*A)_{str}$ and 
  $(\pi^*B)_{str}$ respectively. Then:
     \[   A \cdot B = \langle a,b \rangle. \]
\end{proposition}

\begin{proof}  This proof uses  directly Definition \ref{def:Mumftot}. 

     As $\pi$ is an embedded resolution of $A + B$, the strict transforms 
     $(\pi^*A)_{str}$ and $(\pi^*B)_{str}$ are disjoint. Therefore 
     $(\pi^*A)_{str} \cdot (\pi^*B)_{str}=0$. Using the last condition in the Definition 
     \ref{def:Mumftot} of the total transform of a divisor, we know that 
     $(\pi^*A) \cdot (\pi^* B)_{ex} = (\pi^*A)_{ex} \cdot (\pi^* B) =0$. 
     Combining both equalities, we deduce that:
        \[ \begin{array}{ll}
             A \cdot B &  = (\pi^* A) \cdot (\pi^* B) = \\
                             &  = (\pi^* A) \cdot ( (\pi^* B)_{ex} + (\pi^* B)_{str})  =  \\
                             &  = (\pi^* A) \cdot  (\pi^* B)_{str} =  \\
                             &   = ( (\pi^* A)_{ex} + (\pi^* A)_{str}) \cdot  (\pi^* B)_{str} =  \\
                             &   = (\pi^* A)_{ex}  \cdot  (\pi^* B)_{str} =  \\
                             &  =  (\pi^* A)_{ex}  \cdot  ( \pi^* B -  (\pi^* B)_{ex} ) =  \\
                             & =  -  (\pi^* A)_{ex}  \cdot   (\pi^* B)_{ex}  =  \\
                             &  = -  (-E_a^{\vee}) \cdot  (-E_b^{\vee}) = \\
                             & = \langle a,b \rangle. 
           \end{array} \]
      At the end of the computation we have used the equality $(\pi^* A)_{ex} = - E_a^{\vee}$, 
      which results from the fact that $\pi$ is an embedded resolution of $A$. Indeed, this 
      implies that $((\pi^* A)_{str}  + E_a^{\vee} ) \cdot E_u =0$ for every $u \in \cP(\pi)$, 
      which shows that one has indeed the stated formula for $(\pi^* A)_{ex}$. 
\end{proof}

\section{A reformulation of the ultrametric inequality}
\label{sec:refineq}
\medskip

In this section we explain the notion of \emph{angular distance} on the set of vertices 
of the dual graph of a good resolution of $S$. Theorem \ref{thm:isdist} states 
a crucial property of this distance, relating it to the 
cut-vertices of the dual graph. Then the ultrametric inequality is reexpressed in 
terms of the angular distance. This allows to show that Theorem \ref{thm:ultrambv}  
is a consequence of Theorem \ref{thm:reformangdist}, 
which is formulated only in terms of the angular distance. 
\medskip

Let $\pi : S^{\pi} \to S$ be a good resolution of the normal surface singularity $S$. 
Recall that $\cP(\pi)$ denotes the set of irreducible components of its exceptional 
divisor $E^{\pi}$. Using the notion of bracket from Definition \ref{def:bracket}, one 
may define (see \cite[Sect. 2.7]{GR 17} and  \cite[Sect. 1.2]{GBGPPPR 19}):

\begin{definition}   \label{def:rho}   \index{angular distance} \index{distance!angular}
    The {\bf angular distance} is the function $\rho : \cP(\pi) \times \cP(\pi) \to [0, \infty)$ 
    given by:
       \[  \boxed{\rho(a,b)}   := \left\{ \begin{array}{cl}
                                               - \log \dfrac{\langle a,b\rangle^2}{\langle a, a \rangle \langle b,b\rangle} &  
                                                  \mbox{ if }  a \neq b, \\
                                              0    & \mbox{ if }  a = b. 
                                           \end{array}  \right.\]
\end{definition}

The fact that the function $\rho$ takes values in the interval $[0, \infty)$ is a consequence of 
Lemma \ref{lem:CS}. The attribute ``angular'' was chosen by Gignac and Ruggiero 
because their definition in \cite[Sect. 2.7]{GR 17} 
was more general, applying to any pair of real-valued semivaluations 
of the local ring $\cO_{S, s}$, and that it depended only on those valuations up to homothety, 
similarly to the angle of two vectors.  It is a distance by the following theorem of 
Gignac and Ruggiero \cite[Prop. 1.10]{GR 17} (recall that the notion of vertex separating two other 
vertices was introduced in Definition \ref{def:cut}):

\begin{theorem}  \label{thm:isdist}   
      The function $\rho$ is a distance on the set $\cP(\pi)$. Moreover, for every 
      $a,b,c \in \cP(\pi)$, the following properties are equivalent:
         \begin{itemize} 
             \item one has the equality 
                 $\rho(a,b) + \rho(b,c) = \rho(a,c)$;
             \item  $b$ separates $a$ and $c$ 
                      in the dual graph $\Gamma(\pi)$. 
         \end{itemize}
\end{theorem}

This theorem explains the importance of cut-vertices of the dual graph $\Gamma(\pi)$ for 
understanding the angular distance. 

Theorem \ref{thm:isdist}  is a reformulation of the following theorem, which was first proved by 
in \cite[Prop. 79, Rem. 81]{GBGPPP 18} for arborescent 
singularities, then in \cite[Prop. 1.10]{GR 17} for arbitrary normal surface singularities 
(see also \cite[Prop. 1.18]{GBGPPPR 19} for a slightly 
different proof):

\begin{theorem}  \label{thm:crucineq}   
      Let $a,b,c \in \cP(\pi)$. Then:
        \[ \langle a,b\rangle \langle b,c\rangle \leq \langle b,b\rangle \langle a,c\rangle, \]
     with equality if and only if $b$ separates $a$ and $c$ 
      in the dual graph $\Gamma(\pi)$. 
\end{theorem}

Theorem \ref{thm:crucineq} may be also reformulated in terms of spherical geometry 
using the spherical Pythagorean theorem (see \cite[Prop. 1.19.III]{GBGPPPR 19}).

Using Proposition \ref{prop:intbra} and Definition \ref{def:rho}  of the angular distance, one 
may reformulate in the following way the ultrametric inequality for the restriction 
of the function $u_L$ to a set of three branches:

\begin{proposition}  \label{prop:reform}
    Let $L, A, B, C$ be pairwise distinct branches on $S$. Consider an embedded resolution 
    of their sum and let $E_l, E_a, E_b, E_c$ the irreducible components of 
    its exceptional divisor which intersect the strict transforms of $L, A, B$ 
    and $C$ respectively. Then the following (in)equalities are equivalent: 
        \begin{enumerate}
            \item  $u_L(A, B) \leq \max \{ u_L(A,C), u_L(B,C)\}$. 
            \item \label{point:symint} 
                 $(A \cdot B) \cdot (L \cdot C) \geq \min \{ (A \cdot C) (L \cdot B) , 
                  (B \cdot C) (L \cdot A) \}.$
            \item  $\langle a, b \rangle  \langle l, c \rangle \geq 
                \min \{ \langle a, c \rangle \langle l, b \rangle,  \:   \langle b, c \rangle  \langle l, a \rangle \}.$
            \item $\rho(a,b) + \rho(l, c) \leq \max \{  \rho(a,c) + \rho(l, b),  \rho(b,c) + \rho(l, a) \}. $
        \end{enumerate}
\end{proposition}

We leave the easy proof of this proposition to the reader. It uses the definitions of the 
function $u_L$, of the angular distance, as well as Proposition \ref{prop:intbra}. 
Note that excepted the first one, all 
the inequalities are symmetric in the four branches $L, A, B, C$. The fourth one is a 
well-known condition in combinatorics, whose name was introduced by Bunemann 
in his 1974 paper \cite{B 74}:
 
\begin{definition}   \label{def:4pt}     \index{four points condition}
     Let $(X, \delta)$ be a finite metric space. One says that it satisfies the {\bf four points condition} 
     if whenever $a,b,c,d \in X$, one has the following inequality:
      \[   \delta(a,b) + \delta(c,d) \leq \max \{ \delta(a,c) + \delta(b,d) , \: 
               \delta(a,d) + \delta(b,c) \}  .  \]
\end{definition}

In the same way in which a finite ultrametric may be thought as a special kind of decorated  
\emph{rooted} tree (see Proposition \ref{prop:equivultramtree}), 
a finite metric space satisfying the four points condition may be thought 
as a special kind of decorated \emph{unrooted} tree (see \cite{BD 98}):

\begin{proposition}  \label{prop:satisf4pt}
     The metric space $(X, \delta)$ satisfies the four points condition if and only if $\delta$ 
     is induced by a length function on a tree containing the set $X$ among its set of vertices. 
     If, moreover, one constrains $X$ to contain all the vertices of the tree of valency $1$ or $2$, 
     then this tree is unique up to a unique isomorphism fixing $X$. 
\end{proposition}

Let us introduce supplementary vocabulary in order to deal with the special trees appearing 
in Proposition \ref{prop:satisf4pt}:

\begin{definition}  \label{def:treehull}  \index{tree} 
     Let $X$ be a finite set. An {\bf $X$-tree} is a tree whose set of vertices 
     contains the set $X$ and such that each vertex of valency at most $2$ belongs to $X$.      
    If $(X, \delta)$ is a finite metric space which satisfies the four points condition, 
    then the unique $X$-tree  characterized in Proposition \ref{prop:satisf4pt} 
    is called the {\bf tree hull} of  $(X, \delta)$.   \index{tree!hull}
\end{definition}

The basic idea of the proof of Proposition \ref{prop:satisf4pt} is that an $X$-tree is 
characterized by the shapes of the convex hulls of the quadruples of points of $X$, 
and that those shapes are determined by the cases of equality in the $12$ triangle 
inequalities and the $3$ four points conditions associated to each quadruple. 
In Figure \ref{fig:Possibshape} are represented the five possible shapes. For instance, 
the $H$-shape is the generic one, characterized by the fact that one has no equality 
in the previous inequalities.

 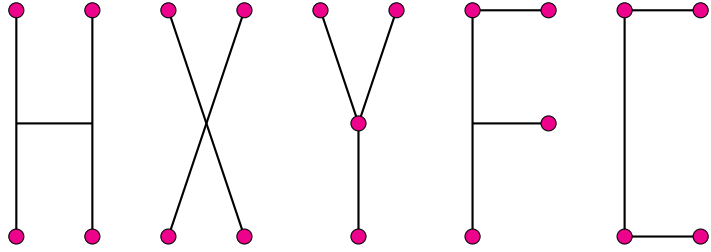
\begin{figure}[h!]
    \begin{center}
\begin{tikzpicture}[scale=0.5]

\begin{scope}[shift={(0,0)}]          
    \draw [-, color=black, thick](0,0) -- (0,6);
     \draw [-, color=black, thick](2, 0) -- (2,6) ;
      \draw [-, color=black, thick](0, 3) -- (2,3) ;
      
      \node[draw,circle, inner sep=2pt,color=black, fill=magenta] at (0,0){};
     \node[draw,circle, inner sep=2pt,color=black, fill=magenta] at (0,6){};
     \node[draw,circle, inner sep=2pt,color=black, fill=magenta] at (2, 0){};
     \node[draw,circle, inner sep=2pt,color=black, fill=magenta] at (2, 6){};
 \end{scope}

 \begin{scope}[shift={(4,0)}]          
    \draw [-, color=black, thick](0,0) -- (2,6);
     \draw [-, color=black, thick](2, 0) -- (0,6) ;
      
      \node[draw,circle, inner sep=2pt,color=black, fill=magenta] at (0,0){};
     \node[draw,circle, inner sep=2pt,color=black, fill=magenta] at (0,6){};
     \node[draw,circle, inner sep=2pt,color=black, fill=magenta] at (2, 0){};
     \node[draw,circle, inner sep=2pt,color=black, fill=magenta] at (2, 6){};
 \end{scope}

 \begin{scope}[shift={(8,0)}]          
    \draw [-, color=black, thick](1,0) -- (1,3);
     \draw [-, color=black, thick](1, 3) -- (2,6) ;
      \draw [-, color=black, thick](1, 3) -- (0,6) ;
      
      \node[draw,circle, inner sep=2pt,color=black, fill=magenta] at (1,0){};
     \node[draw,circle, inner sep=2pt,color=black, fill=magenta] at (1,3){};
     \node[draw,circle, inner sep=2pt,color=black, fill=magenta] at (0, 6){};
     \node[draw,circle, inner sep=2pt,color=black, fill=magenta] at (2, 6){};
 \end{scope}

  \begin{scope}[shift={(12,0)}]          
    \draw [-, color=black, thick](0,0) -- (0,6) -- (2,6);
     \draw [-, color=black, thick](0, 3) -- (2,3) ;
      
      \node[draw,circle, inner sep=2pt,color=black, fill=magenta] at (0,6){};
     \node[draw,circle, inner sep=2pt,color=black, fill=magenta] at (2,6){};
     \node[draw,circle, inner sep=2pt,color=black, fill=magenta] at (0, 0){};
     \node[draw,circle, inner sep=2pt,color=black, fill=magenta] at (2, 3){};
 \end{scope}

  \begin{scope}[shift={(16,0)}]          
    \draw [-, color=black, thick] (2,0) -- (0,0) -- (0,6) -- (2,6);
      
      \node[draw,circle, inner sep=2pt,color=black, fill=magenta] at (0,6){};
     \node[draw,circle, inner sep=2pt,color=black, fill=magenta] at (2,6){};
     \node[draw,circle, inner sep=2pt,color=black, fill=magenta] at (0, 0){};
     \node[draw,circle, inner sep=2pt,color=black, fill=magenta] at (2, 0){};
 \end{scope}       
    
    \end{tikzpicture}
\end{center}
\caption{The possible shapes of an $X$-tree, when $X$ has four elements.} 
 \label{fig:Possibshape}
   \end{figure}

Let us come back to our normal surface singularity $S$. One has the following property 
(see \cite[Prop. 1.24]{GBGPPPR 19}):

\begin{proposition}   \label{prop:anybranchultram}
     Let $\cF$ be a finite set of branches on $S$. If $u_L$ is an ultrametric 
     on $\cF \: \setminus \: \{L \}$ for one branch $L$ in $\cF$, then the same 
     is true for any branch of $\cF$. 
\end{proposition}

By Proposition \ref{prop:reform}, if $u_L$ is an ultrametric 
     on $\cF \: \setminus \: \{L \}$ for one branch $L$ in $\cF$, then 
     one has the symmetric relation (\ref{point:symint}) for every quadruple 
     of branches of $\cF$ \emph{containing $L$}. The subtle point of the proof 
     of  Proposition \ref{prop:anybranchultram}   
     is to deduce from this fact that (\ref{point:symint}) is satisfied by \emph{all} 
     quadruples.

Given Proposition \ref{prop:anybranchultram}, 
it is natural to try to relate the rooted trees associated 
to the ultrametrics obtained by varying $L$ among the branches of $\cF$. 
By looking at quadruples of branches from $\cF$, one may prove using 
Propositions \ref{prop:reform} and \ref{prop:anybranchultram} that:

\begin{proposition}  \label{prop:injrel}
   Let $\cF$ be a finite set of branches on $S$. Consider an embedded resolution of 
   $\cF$ such that the map associating to each branch $A$ of $\cF$ the component 
   $E_a$ of the exceptional divisor intersected by its strict transform is injective. Denote 
   by $\cF^{\pi}$ the set of divisorial valuations $a$ appearing in this way. 
   Then:
       \begin{enumerate}
            \item  The function $u_L$ is an ultrametric on $\cF \: \setminus \: \{L \}$ 
                for some branch $L \in \cF$ if and only if the angular distance $\rho$ 
                satisfies the four points condition in restriction to the set $\cF^{\pi}$. 
            \item  Assume that the previous condition is satisfied. Then the rooted tree associated 
               to $u_L$ on $\cF \: \setminus \: \{L \}$ is isomorphic to the tree hull of 
               $(\cF^{\pi}, \rho)$ by an isomorphism which sends each end marked by a 
               branch $A$ of $\cF$ to the vertex $a$ of the tree hull. 
       \end{enumerate}
\end{proposition}

Proposition \ref{prop:injrel} implies readily that Theorem \ref{thm:ultrambv} 
is a consequence of the following fact (see \cite[Cor. 1.40]{GBGPPPR 19}):

\begin{theorem}   \label{thm:reformangdist}
    Let $S$ be a normal surface singularity. Consider a set $\cG$ of vertices of the dual 
    graph $\Gamma$ of a good resolution $\pi: S^{\pi} \to S$ of $S$. Assume that the convex hull 
     $\mathrm{Conv}_{\cB\cV(\Gamma)}(\cG)$ of $\cG$ in the brick-vertex tree of $\Gamma$ 
     does not contain brick-vertices 
     of valency at least $4$ in $\mathrm{Conv}_{\cB\cV(\Gamma)}(\cG)$. Then 
     the restriction of the angular distance $\rho$ to $\cG$ satisfies the four points condition and 
     the associated tree hull is isomorphic as a $\cG$-tree to $\mathrm{Conv}_{\cB\cV(\Gamma)}(\cG)$.
\end{theorem}

In turn, Theorem \ref{thm:reformangdist} is a consequence of a graph-theoretic result 
presented in the next section (see Theorem \ref{thm:BVuse}).

\section{A theorem of graph theory}
\label{sec:thmgraph}
\medskip

In this final section we state a pure graph-theoretical theorem, which implies 
Theorem \ref{thm:reformangdist} of the previous section. As we explained before, 
that theorem implies in turn our strongest generalization of P\l oski's theorem, 
that is, Theorem \ref{thm:ultrambv}. 
\medskip

Theorem \ref{thm:reformangdist} is a consequence of Theorem \ref{thm:isdist}  and 
of the following graph-theoretic result: 

\begin{theorem}   \label{thm:BVuse}
    Let $\Gamma$ be a finite connected graph and $\delta$ be a distance 
    on the set $V(\Gamma)$ of vertices of $\Gamma$, such that
    for every $a,b,c \in V(\Gamma)$, the following properties are equivalent:
         \begin{itemize} 
             \item one has the equality 
                 $\delta(a,b) + \delta(b,c) = \delta(a,c)$;
             \item  $b$ separates $a$ and $c$  in $\Gamma$. 
         \end{itemize}
    Let $X$ be a set of vertices of $\Gamma$ such that the convex hull 
    $\mathrm{Conv}_{\cB\cV(\Gamma)}(X)$ of $X$ in the brick-vertex tree of 
    $\Gamma$ does not contain brick-vertices of valency at least $4$ in 
    $\mathrm{Conv}_{\cB\cV(\Gamma)}(X)$. Then $\delta$ satisfies 
    the $4$ points condition in restriction to $X$ and the tree hull 
    of $(X, \delta)$ is isomorphic to $\mathrm{Conv}_{\cB\cV(\Gamma)}(X)$ 
    as an $X$-tree. 
\end{theorem}

The idea of the proof of Theorem \ref{thm:BVuse} is to show that, under the given 
hypotheses, the equalities among the triangle inequalities and four points conditions 
are as described by the brick-vertex tree. It is writtend in a detailed way 
in \cite[Thm. 1.38]{GBGPPPR 19}.

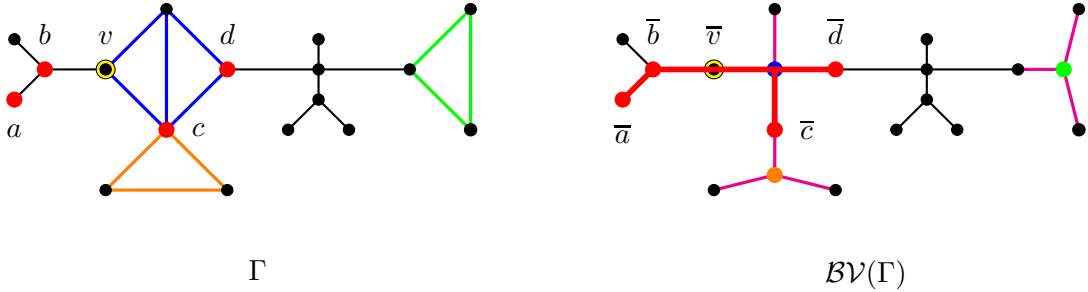
\begin{figure}[h!]
\begin{center}
\begin{tikzpicture}[scale=0.4]

   \draw [-, color=black, thick](0,0) -- (-1,1) ;
 \draw [-, color=black, thick](0,0) -- (-1,-1) ;
  \node[draw,circle, inner sep=1.5pt,color=black, fill=black] at (-1,1){};
  \node[draw,circle, inner sep=2pt,color=red, fill=red] at (-1,-1){};
          \node [below] at (-1,-1.5) {$a$};
 
   \draw [-, color=black, thick](0,0) -- (2,0) ;
   
    \draw [-, color=blue, very thick](4,2) -- (2,0) ;
    \draw [-, color=blue, very thick](4,-2) -- (2,0) ;
    \draw [-, color=blue, very thick](4,-2) -- (4,2) ;
       \draw [-, color=blue, very thick](4,-2) -- (6,0) ;
    \draw [-, color=blue, very thick](6,0) -- (4,2) ;
    
     \draw [-, color=orange, very thick](2,-4) -- (4,-2) ;
    \draw [-, color=orange, very thick](4,-2) -- (6,-4) ;
    \draw [-, color=orange, very thick](2,-4) -- (6,-4) ;
      \node[draw,circle, inner sep=1.5pt,color=black, fill=black] at (4,2){};
  \node[draw,circle, inner sep=2pt,color=red, fill=red] at (4,-2){};
         \node [right] at (4.5,-2) {$c$};
        
  \node[draw,circle, inner sep=2pt,color=red, fill=red] at (0,0){};
          \node [above] at (0,0.5) {$b$};
  
    \node[draw,circle, inner sep=2.5pt,color=black, fill=yellow] at (2,0){};
    \node [above] at (2,0.5) {$v$};
  \node[draw,circle, inner sep=1.5pt,color=black, fill=black] at (2,0){};
    \node[draw,circle, inner sep=1.5pt,color=black, fill=black] at (6,-4){};
    \node[draw,circle, inner sep=1.5pt,color=black, fill=black] at (2,-4){};
    \draw [-, color=black, thick](6,0) -- (9,0) ;
    \draw [-, color=black, thick](9,0) -- (9,1) ;
    \draw [-, color=black, thick](9,-1) -- (9,0) ;
    \draw [-, color=black, thick](9,0) -- (12,0) ;
    
  \node[draw,circle, inner sep=2pt,color=red, fill=red] at (6,0){};
       \node [above] at (6,0.5) {$d$};
  
  \node[draw,circle, inner sep=1.5pt,color=black, fill=black] at (9,0){};
    
      \draw [-, color=black, thick](8,-2) -- (9,-1) ;
    \draw [-, color=black, thick](9,-1) -- (10,-2) ;
  \node[draw,circle, inner sep=1.5pt,color=black, fill=black] at (9,-1){};
  \node[draw,circle, inner sep=1.5pt,color=black, fill=black] at (8,-2){};
  \node[draw,circle, inner sep=1.5pt,color=black,  fill=black] at (10,-2){};
    \node[draw,circle, inner sep=1.5pt,color=black, fill=black] at (9,1){};
    \draw [-, color=green, very thick](12,0) -- (14,-2) ;
    \draw [-, color=green, very thick](12,0) -- (14,2) ;
    \draw [-, color=green, very thick](14,-2) -- (14,2) ;
       \node[draw,circle, inner sep=1.5pt,color=black, fill=black] at (12,0){};
   \node[draw,circle, inner sep=1.5pt,color=black, fill=black] at (14,2){};
  \node[draw,circle, inner sep=1.5pt,color=black, thick, fill=black] at (14,-2){};

  \node [below] at (7,-6) {$\Gamma$};
  

  \begin{scope}[shift={(20,0)}]
  
   \draw [-, color=black, thick](0,0) -- (-1,1) ;
 \draw [-, color=black, thick](0,0) -- (-1,-1) ;
  \node[draw,circle, inner sep=1.5pt,color=black, fill=black] at (-1,1){};
  \node[draw,circle, inner sep=2pt,color=red, fill=red] at (-1,-1){};
          \node [below] at (-1,-1.5) {$\overline{a}$};
 
   \draw [-, color=black, thick](0,0) -- (2,0) ;    
    
    \draw [-, color=magenta, very thick](4,0) -- (2,0) ;
    \draw [-, color=magenta, very thick](4,0) -- (4,-2) ;
     \draw [-, color=magenta, very thick](4,0) -- (4,2) ;  
     \draw [-, color=magenta, very thick](4,0) -- (6,0) ;
     \node[draw,circle, inner sep=2pt,color=blue, fill=blue] at (4,0){};
        
     \draw [-, color=magenta, very thick](4,-3.5) -- (2,-4) ;
    \draw [-, color=magenta, very thick](4,-3.5) -- (4,-2) ;
     \draw [-, color=magenta, very thick](4,-3.5) -- (6,-4) ;
     \node[draw,circle, inner sep=2pt,color=orange, fill=orange] at (4,-3.5){};
    
      \node[draw,circle, inner sep=1.5pt,color=black, fill=black] at (4,2){};
      \node[draw,circle, inner sep=2pt,color=red, fill=red] at (4,-2){};
      \node [right] at (4.5,-2) {$\overline{c}$};
         
  \node[draw,circle, inner sep=2pt,color=red, fill=red] at (0,0){};
          \node [above] at (0,0.5) {$\overline{b}$};
 
          \node[draw,circle, inner sep=2.5pt,color=black, fill=yellow] at (2,0){};
              \node [above] at (2,0.5) {$\overline{v}$};
  \node[draw,circle, inner sep=1.5pt,color=black, fill=black] at (2,0){};
    \node[draw,circle, inner sep=1.5pt,color=black, fill=black] at (6,-4){};
    \node[draw,circle, inner sep=1.5pt,color=black, fill=black] at (2,-4){};
    \draw [-, color=black, thick](6,0) -- (9,0) ;
    \draw [-, color=black,  thick](9,0) -- (9,1) ;
    \draw [-, color=black,  thick](9,-1) -- (9,0) ;
    \draw [-, color=black, thick](9,0) -- (12,0) ;
    
   \node[draw,circle, inner sep=2pt,color=red, fill=red] at (6,0){};
       \node [above] at (6,0.5) {$\overline{d}$};

  \node[draw,circle, inner sep=1.5pt,color=black, fill=black] at (9,0){};
    
    \draw [-, color=black, thick](8,-2) -- (9,-1) ;
    \draw [-, color=black, thick](9,-1) -- (10,-2) ;
    \node[draw,circle, inner sep=1.5pt,color=black, fill=black] at (9,-1){};
  \node[draw,circle, inner sep=1.5pt,color=black, fill=black] at (8,-2){};
  \node[draw,circle, inner sep=1.5pt,color=black,  fill=black] at (10,-2){};
    \node[draw,circle, inner sep=1.5pt,color=black, fill=black] at (9,1){};

     \draw [-, color=magenta, very thick](13.5,0) -- (12,0) ;
    \draw [-, color=magenta, very thick](13.5,0) -- (14,-2) ;
     \draw [-, color=magenta, very thick](13.5,0) -- (14,2) ;
    \node[draw,circle, inner sep=2pt,color=green, fill=green] at (13.5,0){};
    
       \node[draw,circle, inner sep=1.5pt,color=black, fill=black] at (12,0){};
   \node[draw,circle, inner sep=1.5pt,color=black, fill=black] at (14,2){};
  \node[draw,circle, inner sep=1.5pt,color=black, thick, fill=black] at (14,-2){};

     \node [below] at (7,-6) {$\cB\cV(\Gamma)$};
     
      \draw [-, line width=0.75mm, color=red] (-1,-1)--(0,0)--(6,0);
   \draw [-, line width=0.75mm, color=red] (4,0)--(4,-2);

   \end{scope}
 
 \end{tikzpicture}
\end{center}
\caption{A convex hull of four vertices}
\label{fig:Computineq}
\end{figure}

\begin{example}  \label{ex:graphthm}
    Let us consider again the connected graph $\Gamma$ of Example \ref{ex:brickvertreeex}. 
    Look at its vertices $a,b,c,d$ shown on the left of Figure \ref{fig:Computineq}. 
    The corresponding vertices $\overline{a}, \overline{b}, \overline{c}, \overline{d}$ 
    of the brick-vertex tree $\cB \cV(\Gamma)$ are shown on the right side of the figure. 
    Denoting $X := \{a,b,c,d\}$, the convex hull $\mbox{Conv}_{\cB\cV(\Gamma)}(X)$  
    is also drawn on the right side using thick red segments. 
     We see that the hypothesis of Theorem \ref{thm:BVuse} about the valencies of 
     brick-vertices is satisfied, as the only 
     brick-vertex contained in $\mbox{Conv}_{\cB\cV(\Gamma)}(X)$ is of valency $3$ 
     in this convex hull. 
     
     As shown by the $F$-shape of $\mbox{Conv}_{\cB\cV(\Gamma)}(X)$, one 
     should have the following equalities and inequalities in the four points conditions 
     concerning $X$:
          \begin{equation} \label{eq:eqineq} 
                \delta(a,d) + \delta(b,c) =   \delta(a,c) + \delta(b,d) >  \delta(a,b) + \delta(c,d).  
          \end{equation}
     Let us prove that this is indeed the case. Consider the cut vertex $v$ of $\Gamma$  shown 
     on the left side of Figure \ref{fig:Computineq}. As it separates $a$ from $d$, 
     we have the equality $\delta(a,d) = \delta(a,v) + \delta(v,d)$. As $v$ 
     does not separate $a$ from $b$, we have the strict inequality 
     $\delta(a,v) + \delta(b,v) >  \delta(a,b)$. Using similar equalities and inequalities, we get:
       \[ \begin{array}{l} 
                \delta(a,d) + \delta(b,c) =  \\
                = (\delta(a,v) + \delta(v,d)) + ( \delta(b,v) + \delta(v,c)  )  = \\
                =  (\delta(a,v) + \delta(v,c)) + ( \delta(b,v) + \delta(v,d)  )  = \\
                =   \delta(a,c) + \delta(b,d) =  \\
                =  (\delta(a,v) + \delta(b,v)) + ( \delta(v,d) + \delta(v,c)  )  = \\
                >  \delta(a,b) + \delta(c,d). 
           \end{array} \]
     The (in)equalities (\ref{eq:eqineq}) are proved. 
     
     One proves similarly the triangle equalities $\delta(a,b) + \delta(b,c) =  \delta(a,c)$, 
     $\delta(a,b) + \delta(b,d) =  \delta(a,d)$ and the fact that one has no equality 
     among the triangle inequalities concerning the triple $\{a,c,d\}$, which shows that 
     the tree hull of $(X, \delta)$ has indeed an $F$-shape, with the vertices 
     $a,b,c,d$ placed as in $\mbox{Conv}_{\cB\cV(\Gamma)}(X)$.
\end{example}

\printindex

\medskip
\end{document}